\documentclass[12pt,a4paper,final]{article}
\usepackage[margin=25.4mm]{geometry}
\usepackage[utf8]{inputenc}

\usepackage{amsmath,amsfonts,amssymb,bef_alex,%
            datetime,todonotes,bm,relsize,tikz}
\usepackage[mathscr]{euscript}
\renewcommand{\ti}{{\times}}
\def\FG{\bm}

\numberwithin{equation}{section}
\numberwithin{figure}{section}

\usepackage[notref,notcite,color]{showkeys}

\makeatletter
\renewcommand*\env@cases[1][1.2]{%
  \let\@ifnextchar\new@ifnextchar
  \left\lbrace
  \def\arraystretch{#1}%
  \array{@{}c@{\quad}l@{}}%
}
\makeatother

\newcommand{\ThreeVect}[3]{\!\bma{@{}c@{}} #1\\[-0.1em] #2\\[-0.1em] #3 \ema\!}
\newcommand{\LB}{\lambda_\rmB}
\newcommand{\mfC}{\mathfrak C}
\newcommand{\mfU}{\mathfrak U}
\newcommand{\ALPHA}{\rma_{\rml\rmo}}
\newcommand{\BETA}{\rma_{\rmu\rmp}}
\newcommand{\dom}{\mathop{\mafo{dom}}}
\newcommand{\STEP}[1]{\underline{\em #1}}

\newcommand{\TODO}[1]{}

\begin{document}

\title{Existence of similarity profiles for \\  diffusion equations and systems%
    \thanks{Research partially supported by DFG via SFB\,910 ``Control of self-organizing
nonlinear systems'' (project no.\,163436311), subproject A5 ``Pattern
formation in coupled parabolic systems''.} }

\author{Alexander Mielke\thanks{Weierstraß-Institut f\"ur Angewandte
   Analysis und Stochastik, Mohrenstr.\,39, 10117 Berlin and  Humboldt
 Universit\"at zu Berlin, Institut f\"ur Mathematik, Rudower Chaussee 25, 12489
 Berlin, Germany.}{}\ 
\ 
and 
Stefanie Schindler\thanks{Weierstraß-Institut f\"ur Angewandte
   Analysis und Stochastik, Mohrenstr.\,39, 10117 Berlin, Germany.}
}

\date{5 April 2023}
 
\maketitle

\TODO{DRAFT  \tiny \usdate\today, \currenttime}

\begin{abstract}
  \textbf{Abstract.} We study the existence of similarity profiles for
  diffusion equations and reaction diffusion systems on the real line, where
  the different nontrivial limits are imposed for $ x \to -\infty$ and
  $x \to +\infty$. These similarity profiles solve a coupled system of nonlinear ODEs
  that can be treated by monotone operator theory.
\end{abstract}

 \smallskip


\TODO{To be cited \cite{GaJoRa22ASSD} !! Here or in second paper. 
\\
To be cited: \cite{Sham73CDD, Sham73CDD2, Sham76CDD3}, \cite{VanPel77CSSN} 
}

\section{Introduction}
\label{se:intro}

Similarity profiles play an important role in the longtime behavior of
nonlinear diffusion problems as well as in certain reaction-diffusion systems,
if we consider problems posed on the whole space $\Omega=\R^d$. For simplicity
we treat the one-dimensional case $\R^d=\R$ only and leave the case
$d>1$ for subsequent work.

We consider a system of coupled diffusion equations on the real line $\Omega=\R^1$:
\begin{equation}
  \label{eq:DiffSyst}
  \dot\bfu = \big( \bfA(\bfu)\big)_{xx} \quad \text{for  }t>0, \ x\in \R, \qquad \bfu(t,\pm
  \infty) = \bfU_\pm \text{ for }t>0, 
\end{equation}
where $\dot \bfu= \bfu_t$ and $\bfA:\R^m\to \R^m$ is a smooth monotone mapping. Here
$\bfU_\pm=\bfu(t,\pm\infty):= \lim_{x\to \pm \infty} \bfu(t,x)$  are nontrivial
boundary conditions, namely $\bfU_-$ for $x \to -\infty$ and $\bfU_+\neq \bfU_-$ for $x
\to + \infty$. 

The aim of this paper concerns the existence of self-similar solutions for this
system. As $\bfu$ cannot be scaled because of the fixed boundary conditions,
we use the parabolic similarity coordinates $\tau=\log(t{+}1)$ and $y=
x/(t{+}1)^{1/2}$. 
Setting $\wt\bfu(\tau,y) = \bfu(t,x)$, the transformed equation reads  
\begin{equation}
  \label{eq:DiffSystSim}
  \wt\bfu_\tau = \big(\bfA(\wt\bfu)\big)_{yy}+ \frac{y}2 \wt\bfu_y \ \text{ for } \tau >0,  \ y \in
  \R, \quad   \wt\bfu(\tau,\pm \infty) = \bfU_\pm .  
\end{equation}
A stationary solution $\bfU \in \rmB\rmC^2(\R;\R^m)$ of this equation is called
\emph{similarity profile} as it gives rise to a self-similar solution
\[
\bfu(t,x)= \bfU\big(x/(t{+}1)^{1/2}\big) 
\]
of the original diffusion system  \eqref{eq:DiffSyst}.  

Our main goal is to show that the following boundary value problem for a
second order ODE in $\R^m$ has a (unique) solution:
\begin{equation}
  \label{eq:I.ProfileEqn}
  \big( \bfA(\bfU(y))\big)'' + \frac y2 \,\bfU'(y)=0 \ \text{ for }y\in \R,
  \qquad \lim_{y\to \pm \infty} \bfU(y) = \bfU_\pm. 
\end{equation}
This equation is called the \emph{profile equation}. 

Our main existence result of self-similar profiles is formulated in Theorem
\ref{th:VectValProf} and concerns a vector-valued generalization of the scalar
monotonicity result developed in \cite[Thm.\,3.1]{GalMie98DMSS}.  The advantage
of using monotonicity in contrast to the ODE-type arguments in previous works,
see e.g.\ \cite{Sham76CDD3, VanPel77CSSN}, is that we can also handle the
vector-valued case.

Section \ref{se:VectorValued} provides a careful theory on existence and
uniqueness of similarity profiles solving \eqref{eq:I.ProfileEqn}. In
particular, we show that the solutions and their derivatives can be estimated
in a linear way by $\Delta_\pm:= \big|\bfU_+-\bfU_-\big|$ with prefactors that
are given explicitly in terms of the constants $\delta, \ALPHA, \BETA=
\mafo{Lip}(\bfA)$, see \eqref{eq:Cond.bfA}, where the crucial assumption is
the monotonicity $\langle \bfA(\bfu){-}\bfA(\bfw), \bfu{-}\bfw\rangle \geq \ALPHA
|\bfu{-}\bfw|^2\geq 0$.    

In Section \ref{se:ScalProfiles} we specialize to the scalar case and improve
the estimates significantly. We show monotonicity of the profile $U:\R\to [U_-,U_+]$
and exponential decay of the flux $Q(y) = A'(U(y))U'(y)$, namely $0\leq Q(y) \leq
\ee^{-y^2/(4D^*)}$, where $D^*=\max\bigset{A'(u)}{u \in [U_-,U_+]}$. In particular, we
allow for degenerate cases where $A'(U_+)=0$ and $A'(U_-)=0$, which may lead to the
case $U(y)=U_*$ for all $y\geq y_+^*$ if $A'(u)=\mathscr{O} (U_+{-}u)$ for $u\nearrow U_+$.       

In Section \ref{se:PME} we study the stability of the similarity profile $U$ as
steady solution of the parabolically rescaled diffusion equation 
\begin{equation}
  \label{eq:I.resPME}
  u_\tau = \big( A(u)\big)_{y y} + \frac y2 \,u_y \quad \text{for } (t,y)\in
  {]0,\infty[} \ti \R, \qquad u(\tau,\pm\infty) = U_\pm.
\end{equation}
For this, we consider relative entropies of the form
\[
\calH_\phi(u(\tau))= \int_\R \phi(u(\tau,y)/U(y)) \,U(y)\,\dd y 
\]
for suitable convex functions $\phi$ with $\phi(1)=0=\phi'(1)$. Theorem
\ref{th:ExpDecay.CA} considers the case of a general $A$ with $|A''(u)|\leq C_\rmA
<\infty$ and Theorem \ref{th:PME.Cvg} treats the porous medium equation
with $A(u)=u^m$. In both cases we provide conditions on $U$ that allow us to
conclude a global decay estimate in the Hellinger distance
\[
c_0 \big\| \sqrt{u(\tau,\cdot)} - \sqrt{U}\big\|_{\rmL^2(\R)}^2 \leq 
\calH_\phi(u(\tau))\leq \ee^{-\Lambda \tau}
\calH_\phi(u(0)) \quad \text{for all }\tau>0,
\]
with $\Lambda = \frac12 - \mathscr{O}(|U_+{-}U_-|)\leq 1/2$. In particular, for the flat profiles
$U\equiv U_\pm$ we always obtain the trivial decay like $\ee^{-\tau/2}$ which
is induced by the drift term $\frac y2\,u_y$ only.  We also refer to
\cite{VanPel77ABSN} for convergence results to self-similar profiles, but they
are quite different and rely on comparison principle arguments, whereas our
entropy approach can be applied to systems as well, see
\cite{MiHaMa15UDER,MieMit18CEER,MieSch22?CSSP}.

Section \ref{se:ApplRDS} can be seen as a preparation for the theory in
\cite{MieSch22?CSSP} that is concerned with reaction-diffusion systems of the
type 
\[
\dot\bfc = \bfD \bfc_{xx} + \bfR(\bfc), \qquad \bfc(t,\pm\infty) = \bfC_\pm,
\]
where we impose nontrivial boundary conditions at $x=\pm \infty$. To study
the diffusive mixing as introduced in \cite{GalMie98DMSS}, we transform into
parabolic similarity coordinates as for \eqref{eq:DiffSystSim} and obtain 
\[
\bfc_\tau = \bfD \bfc_{yy} + \frac y2 \bfc_y + \ee^\tau\, \bfR(\bfc), \qquad
\bfc(\tau,\pm\infty) = \bfC_\pm. 
\]
While in \cite{MieSch22?CSSP} the term $\ee^\tau\, \bfR(\bfc)$ is treated in
full generality, we look here at the simplified model where we assume
$\bfR(\bfc)=0$ as an algebraic constraint and replace the limit
``$\infty\cdot \bm0$'' of ``$\ee^\tau\cdot \bfR(\bfc)$'' for $\tau\to \infty$ by a
vector-valued Lagrange multiplier $\bflambda$ lying in the span of
$\bfR(\cdot)$ (the stoichiometric subspace, see Section \ref{su:ReducedSyst}).

The set of equilibria $\bigset{\bfc\in \R^{i_*}}{ \bfR(\bfc)=0 }$   is
parametrized in the form $\bfc = \Psi(\bfu)$ for $\bfu\in \R^m$ such that $\bbQ
\Psi(\bfu)=\bfu$ for a suitable linear stoichiometric mapping $\bbQ$. This leads to
the reduced parabolic equation 
\[
\bfu_\tau =  \big(\bfA(\bfu)\big)_{yy} + \frac y2\bfu \quad \text{with } \bfA(\bfu) = \bbQ
\bfD \Psi(\bfu).
\] 
In Section \ref{se:ApplRDS} we provide several examples in which we are able to
specify conditions on the reactions and the diffusion constants in
$\bfD=\mafo{diag}(d_j)$ that guarantee that $\bfu \mapsto \bfA(\bfu)$ is indeed
monotone and satisfies the assumptions of the main existence result for self-similar
profiles $\bfU$ solving \eqref{eq:I.ProfileEqn}. In
particular, Section \ref{su:OneReact3Spec} considers a case with three species,
i.e.\ $\bfc\in \R^3$ and one reaction, such that $\bfu\in \R^2$ is
vector-valued.  

Section \ref{se:VariousSyst} provides two more systems where self-similar
profiles are important to describe the longtime asymptotics. First, we recall
the results in \cite{BriKup92RGGL, GalMie98DMSS} which establish diffusive
mixing for roll pattern in the Ginzburg-Landau equation with real
coefficients. Secondly, we comment on the recently established system of
degenerate parabolic equations that includes the porous medium equation and
is expected to have a rich structure of self-similar profiles, see \cite{Miel23TCDP}.

\section{Vector-valued self-similar profiles}
\label{se:VectorValued}

To provide a suitable functional analytical framework for our existence and
uniqueness theory, we set 
\[
\ol\bfu_\pm(y):= 
\begin{cases}
\bfU_\pm& \text{for } {\pm} y >0,\\[-0.25em]
\frac12\big(\bfU_-{+}\bfU_+\big) &\text{for } y=0.
\end{cases}
\]
In the following, we give a weak version of the profile equation
\eqref{eq:I.ProfileEqn}. 
We say that $\bfU \in \rmL^2_\mafo{loc}(\R;\R^m)$ is a \emph{stationary profile} for
\eqref{eq:DiffSystSim} if 
\begin{subequations}
  \label{eq:WeakStatProf}
\begin{align}
 \label{eq:WeakStatProf.a}
&\exists\, \bfv \in \rmH^1(\R;\R^m): \quad \bfU= \ol\bfu_\pm+ \bfv'  \text{ and } 
\\
 \label{eq:WeakStatProf.b}
&  \forall\, \bfpsi\in \rmC^2_\rmc(\R;\R^m):\ \ 
\int_\R\!\Big(\bfA(\bfU(y))\cdot \bfpsi''(y) - \bfU(y) \cdot \big( \frac
y2\bfpsi(y)\big)'  \Big) \dd y =0. 
\end{align}
\end{subequations}
In this formulation $\bfU$ does not need to have any derivative and may be even
discontinuous. We will see that this weak form is important because in
degenerate cases the solution $\bfu$ has low regularity, while we are still
able to proof existence of solutions. For instance, in the very degenerate case
$\bfA(\bfU_+)=\bfA(\bfU_-)$ (which is still consistent with the monotonicity
desired below, but gives
$\rmD\bfA\big((1{-}\theta) \bfU_- {+} \theta\bfU_+\big)(\bfU_+{-}\bfU_-)=0$ for
all $\theta\in [0,1]$), we see that the piecewise constant function
$\bfU=\ol\bfu_\pm$ is a stationary profile solving \eqref{eq:WeakStatProf}. 

We will see later in the Sections \ref{se:ScalProfiles} and \ref{se:PME} 
that the scalar porous medium equation with $\bfA(u)=\frac1m u^m \in \R^1$ 
and $m>1$, leads in the case $U_- = 0$ to
profiles $U\in \rmB\rmC^0(\R)$ with $U(y)=0 $ for all $y \leq y_*<0$ and
$U(y)=c (y-y_*)^{1/(m-1)} +$h.o.t.\ for $y\to y^+$. Hence, $U$ is not twice
differentiable for $m\geq 2$ and $U'$ does not lie in $\rmH^1_\text{loc}(\R)$ for
$m\geq 3$.

Moreover, the requirement \eqref{eq:WeakStatProf.a} is slightly stronger than
asking for $\bfU-\ol\bfu_\pm\in \rmL^2(\R;\R^m)$. Indeed, using the embedding
$\rmH^1(\R;\R^m) \subset \rmC^0_0(\R;\R^m)$ (space of continuous and
decaying functions), \eqref{eq:WeakStatProf.a} implies that the following
improper integral exists:
\begin{equation}
  \label{eq:ImproperInt}
\int_\R \big(\bfU{-}\ol\bfu_\pm\big)\dd y = \lim_{a,b\to \infty} \int_{-a}^b
 \big(\bfU{-}\ol\bfu_\pm\big)\dd y = \lim_{a,b\to \infty}(\bfv(b){-}\bfv(-a)) =0.
\end{equation}

In the following example of linear equations we provide explicit solutions in
terms of vector-valued error functions (integrals of Gaussians). We especially
address the case of degenerate $\bfA(\bfu)=\bbA \bfu$, where $\bbA$ has purely
imaginary eigenvalues, in that case $\bfU$ may be discontinuous or may converge
to $\ol\bfu_\pm$ only like $\mathscr{O}(1/|y|)$. Moreover, we address the approximation
of $\bbA$ by the regular case $\bbA_\eps= \bbA+\eps\bbI$, which will be done in
the proof of the main existence result in Theorem \ref{th:VectValProf}, see
Step 5 there. 

\begin{example}[Linear, vector-valued case]
\label{ex:LinVect}
We consider the case $\bfA(\bfu)=\bbA \bfu$ where the matrix $\bbA\in \R^{m\ti
  m}$ is monotone, i.e.\ $\bfv\cdot \bbA\bfv\geq \ALPHA |\bfv|^2$ with
$\ALPHA\geq 0$. 

(I) At first, let $\ALPHA>0$ such that $\bbA^{-1}$ exists and all
its eigenvalues have positive real parts. From $\bbA\bfU'' + \frac y2 \bfU'=0$
we easily find $\bfU'(y)= \ee^{-y^2 (4\bbA)^{-1}} \bfU'(0)$. Using
$\int_\R \ee^{-y^2(4\bbA)^{-1}} \dd y = (4\pi\bbA)^{1/2} $ (here $\bbA^{1/2}$
is the root with eigenvalues satisfying $\left|\arg \lambda\right|< \pi/4$), we
find the profile connecting $\bfU_-$ and $\bfU_+$ in the form
\begin{equation}
  \label{eq:ExplicLinCase}
  \bfU(y)= \bfU_- + \int_{-\infty}^y \frac1{\sqrt{4\pi}}\, \bbA^{-1/2}\,
\ee^{-\eta^2(4\bbA)^{-1}} \dd \eta \:\big(\bfU_+{-}\bfU_-\big) .
\end{equation}

(II) The above formula can also be extended to the case $\ALPHA=0$, where
$\bbA^{-1}$ may no longer exist. For this it suffices to replace $\bbA$ by
$\bbA_\eps=\bbA{+}\eps I$ and take the limit $\eps\to 0^+$. Indeed, if $\bbA$
has a single real eigenvalue $\lambda=0$, then $\bbA_\eps$ has the eigenvalue
$\lambda_\eps=\eps$. Using a suitable basis, it suffices to observe that
$ \int_{-\infty}^y \frac1{\sqrt{4\pi \lambda_\eps}}\,
\ee^{-\eta^2(4\lambda_\eps)^{-1}} \dd \eta= \Phi(y/\sqrt\eps)$ converges to $0$
for $y <0$ and to $1$ for $y>1$. Thus,
$\bfU_\eps= \bfU_-+\Phi(y/\sqrt\eps)(\bfU_+{-}\bfU_-) $ converges to the piecewise 
constant limit $\ol\bfu_\pm$.

(III) If $\bbA$ has a single pair of purely imaginary eigenvalues $\pm \ii \omega$ with
$\omega >0$, then the
limit procedure leads to the linear ODE
\[
\bbA_\eps \bfU''_\eps + \frac y2 \bfU'_\eps = 0 \quad 
\text{with } \bbA_\eps= \binom{\eps\  {-} \omega}{\omega \quad \eps}.
\]
Turning the vector $\bfU_\eps=(U^1_\eps,U^2_\eps)$ into a complex number
$U_\eps=U_\eps^1 {+} \ii U_\eps^2 \in \C$, we have to solve
$2\lambda_\eps U''_\eps +y U'_\eps =0$ with $\lambda_\eps=\eps {+} \ii
\omega$. Of course, \eqref{eq:ExplicLinCase} holds again but now in (scalar)
complex numbers, and the integrand in \eqref{eq:ExplicLinCase}  (which equals
$U'_\eps$ up to the factor $U_+{-}U_-$) reads
\[ 
\frac1{\sqrt{4\pi \lambda_\eps}} \:
\ee^{\ii \eta^2 \omega/(4\eps^2{+}4\omega^2)}\: 
\ee^{- \eta^2 \eps/(4\eps^2{+}4\omega^2)}.
\]
Hence, for $\eps\to 0^+$ the exponential decay of the integrand is lost, but
the improper integrals for $\eta\in {]{-}\infty,y[}$ still have a good limit
because of the increasing oscillations as in the Fresnel integrals
$\int_\R \ee^{\ii \eta^2/(4\omega)} \dd \eta= \sqrt{2\pi\omega} (1{+}\ii)$.  We
obtain the expansion
\[
U_0(y) = U_- - \frac{\ii \sqrt 2\,\omega}{y} \, \ee^{\ii y^2/(4\omega)} +
\mathscr{O}(1/|y|^3) \quad \text{for } y\to -\infty.
\]
Clearly, $U_0$ can be decomposed into $U_0(y) = \ol u_\pm(y) + v'(y)$ with $v\in
\rmH^1(\R;\C)$ where $|v(y)| \leq C/(1{+}y^2)$ and $|v'(y)|\leq
C/(1{+}|y|)$. Moreover, the improper integral 
$\int_\R y(U_0{-} \ol u_\pm)\dd y $ exists and equals $\ii \omega(U_-{-}U_+)$. 
\end{example}

For the proof of the following result, we introduce a smoothened version of the
function $\ol\bfu_\pm$ by fixing an interpolating function
$\chi\in \rmC^\infty(\R; [-1,1])$ satisfying
\[
\chi(y)=\pm 1\quad \text{for }\pm y\geq 1 \qquad \text{and} \qquad
\chi(-y)=-\chi(y).
\]
For given $\bfU_-,\bfU_+\in \R^m$ and a parameter $\BETA>0$, which will be
specified below, we define the interpolation functions $\wt\bfu_\pm
\in \rmC^\infty(\R;\R^m)$ via 
\begin{equation}
  \label{eq:Cond.wt.bfu*}
  \wt\bfu_\pm(y) =\frac{1{-}\chi\big(y/\sqrt{\BETA}\big)}2 \,\bfU_- 
  + \frac{1{+}\chi(y/\sqrt{\BETA})}{2} \, \bfU_+ 
\end{equation}
such that $\wt\bfu_\pm(y)=\ol\bfu_\pm(y)$ for $\vert y \vert \geq 1/\sqrt{\BETA}$. In
the sequel, $C_\chi$ will denote (possibly different) 
constants that depend only on $\chi$, and,
thus, can be seen as universal constants that are independent of the data
$\bfA$ and $\bfU_\pm$ of our problem. For example, the $\rmL^2$ norm of
$\wt\bfu'_\pm$ and $\wt\bfu''_\pm$ scale as follows:
\begin{equation}
  \label{eq:wtu.pri.BETA}
  \|y^j\wt\bfu'_\pm\|_{\rmL^2} \leq C^\chi_1 \BETA^{(2j-1)/4} \Delta_\pm \ \text{ and
  } \ 
\| \wt\bfu''_\pm\|_{\rmL^2} \leq C^\chi_2 \BETA^{-3/4} \Delta_\pm \ 
\text{ with }
\Delta_\pm : =| \bfU_+{-}\bfU_-|.
\end{equation}
The proof of the following result is based on \cite[Thm.\,3.1]{GalMie98DMSS}, 
which exploits monotonicity arguments to obtain existence and uniqueness.
Here we generalize this approach to the vector-valued case and provide a careful 
bookkeeping of constants in the a priori estimates.

\begin{theorem}[Existence of similarity profiles]
\label{th:VectValProf}
Let $\bfA\in \rmC^1(\R^m;\R^m) $ satisfy 
\begin{subequations}
\label{eq:Cond.bfA}
\begin{align}
\label{eq:Cond.bfA.a}
&\exists\, \BETA>0\ 
\forall\, \bfu,\bfw\in \R^m:\
\big| \bfA(\bfu)- \bfA( \bfw) \big| \leq \BETA |\bfu{-}\bfw|,
\\
\label{eq:Cond.bfA.b}
& \exists\, \ALPHA \geq 0\ \forall\, \bfu,\bfw\in \R^m:\
\langle \bfA(\bfu){-}\bfA(\bfw), \bfu {-}\bfw \rangle \geq \ALPHA
|\bfu{-}\bfw|^2,
\\
\label{eq:Cond.bfA.c}
& \exists\, \delta \geq 0\ \forall\, \bfu, \bfv \in \R^m:\quad 
  \langle \bfv , \rmD \bfA(\bfu) \bfv \rangle \geq \delta\, 
  \big| \rmD \bfA(\bfu) \bfv \big|^2. 
\end{align}
\end{subequations}
If $ \ALPHA {+} \delta > 0 $, then for each pair $(\bfU_-,\bfU_+) \in \R^m\ti \R^m$
there exists a unique 
stationary profile $\bfU=\ol\bfu_\pm+\bfv'$ satisfying \eqref{eq:WeakStatProf}
and the a priori estimate 
\begin{equation}
  \label{eq:APrioriEst}
  \ALPHA\, \|\bfU'\|^2_{\rmL^2} + \|\bfU{-}\wt\bfu_\pm\|^2_{\rmL^2} +
  \frac1{\BETA}\,\| \bfv\|^2_{\rmH^1} \leq C_\chi \BETA^{1/2} \Delta_\pm^2 . 
\end{equation}
Moreover, the flux $\bfq(y)=\big( \bfA(\bfU(y) \big)'=\rmD\bfA (\bfU(y)) \bfU'(y)$
satisfies the pointwise estimate 
\begin{equation}
  \label{eq:FluxEst}
  |\bfq(y)|=\big|\bfA(\bfU)'\big| \leq \ee^{-\delta y^2/4} \,
  C_\chi\,\BETA^{1/2} \, \Delta_\pm  \quad \text{for all } y \in \R. 
\end{equation}
For $\delta>0$ we have the integral relations 
\begin{equation}
  \label{eq:Rel.U.olUpm}
\begin{aligned}
&\int_\R\! \big( \bfU(y)-\ol\bfu_\pm(y)\big) \dd y = 0\in \R^m 
\ \text{ and } \ \\
&
\int_\R\! y\,  \big(\ol\bfu_\pm(y) - \bfU(y) \big) \dd y 
 = \bfA(U_+) - \bfA(U_-) \in \R^m. 
\end{aligned}
\end{equation}

If\/ $\rmD\bfA(\bfu)\in \R^{m\ti m}$ is invertible (as is always the case for
$\ALPHA>0$), then $\bfU\in \rmB\rmC^0(\R;\R^m)\cap
\rmH^1_\mafo{loc}(\R;\R^m) $. If $\bfA$ additionally satisfies 
\begin{equation}
  \label{eq:bfA.smooth}
\bfA\in \rmC^k_\mafo{loc}(\R^m;\R^m) \text{ for }k\in \N\quad \text{and} \quad 
\forall\, y\in \R :\  \rmD\bfA(\bfU(y))\in \R^{m\ti m} \text{ is
  invertible},
\end{equation}
then the profile $\bfU$ satisfies $\bfU  \in \rmB\rmC^k(\R;\R^m)$. 
\end{theorem}
Before providing the proof we remark that conditions \eqref{eq:Cond.bfA.a} and
\eqref{eq:Cond.bfA.b} imply condition \eqref{eq:Cond.bfA.c} with
$\delta=\ALPHA/(\BETA)^2$. However, $\delta \gg\ALPHA/(\BETA)^2$ is possible,
and interesting cases occur for $\ALPHA=0$ and $\delta=1/\BETA >0$, which is
the case for the scalar porous medium equation in Section \ref{se:PME}.

We emphasize that an important point in the proof is the exploitation of the term
$\frac12 y {\vdot} \bfu'$, which generates strict monotonicity and an a priori estimate
for $\Vert \bfU {-} \wt \bfu_{\pm} \Vert_{\rmL^2} $ independent of $\bfA$, see
\eqref{eq:APrioriEst}.\bigskip

\noindent
\begin{proof} Throughout this proof all constants $C_\chi$
  only depend on $\chi$, which is kept fixed, whereas the dependence on
  $\Delta_\pm =|\bfU_+{-}\bfU_-|$ and $\bfA$ (via $\ALPHA$, $\BETA$, and
  $\delta$) will be given explicitly.

We first treat the nondegenerate case $\ALPHA>0$. There we obtain a suitable
maximally strictly monotone operator $\calA$ that provides existence and
uniqueness of solutions. The case $\ALPHA=0$ is treated by regularizing $\bfA$
to $\bfA_\eps(\bfu)=\bfA(\bfu)+ \eps\bfu$, which gives $\ALPHA{}_\eps=\eps>0$
and solutions $\bfU_\eps$. Using $\eps$-independent a priori estimates for
$\| \bfU_\eps - \wt\bfu_\pm\|_{\rmL^2}$ we obtain a weak limit $\bfU$ which is
the desired profile. \bigskip

\STEP{Step 1. Preparations:} 
We proceed as is \cite[Thm.\,3.1]{GalMie98DMSS} and search for $\bfU$ in the
form $\bfU(y)= \wt\bfu_\pm(y)+ \bfv'(y)$ with $\bfv\in \rmH^1(\R;\R^m)$.
Inserting the ansatz for $\bfu$ into the stationarity equation, we obtain (in
$\rmH^{-2}(\R;\R^m)$) the relation 
\[
0= \bfA(\wt\bfu_\pm{+}\bfv')'' + \big(\frac y2(\wt\bfu_\pm{+}\bfv')\big)' - \frac
12(\wt\bfu_\pm{+}\bfv'). 
\]
This equation can be integrated with respect to $y$ yielding the
relation (in $\rmH^{-1}(\R;\R^m)$)
\begin{equation}
  \label{eq:IntegrStatEqn}
  0= \bfA(\wt\bfu_\pm{+}\bfv')' + \frac y2\,\bfv'  -\frac12 \bfv +
  \bfg, \quad \text{where } \bfg(y)=\int_{-\infty}^y \frac\eta2
  \wt\bfu'_\pm(\eta)\dd \eta. 
\end{equation}
By construction, we have $\bfg \in \rmC^\infty_\rmc(\R;\R^m)$ and $\bfg(y)=0$ for
$|y|\geq 1/\sqrt{\BETA}$ (use that $\wt\bfu'_\pm$ is even, see
\eqref{eq:Cond.wt.bfu*}). Moreover, by the scaling of $\wt\bfu_\pm$ via $\BETA$
one obtains 
\begin{equation}
  \label{eq:bfg.BETA}
  \| \bfg\|_{\rmL^2(\R)} \leq C_\chi \BETA^{3/4}.
\end{equation}

\STEP{Step 2. Monotone operator theory for $\ALPHA>0$:} We set
$\bfH:= \rmH^1(\R;\R^m)$, which gives $\bfH^*=$ $\rmH^{-1}(\R;\R^m)$, and
define the monotone operator $\calA: \dom(\calA) \subset \bfH\to \bfH^*$ via
\[
\dom(\calA):=\bigset{\bfv\in \rmH^1(\R;\R^m)}{ y\bfv'(y) \in \bfH^*} 
\text{ and }  \calA(\bfv):= - \big(\bfA(\wt\bfu_\pm{+}\bfv')\big)' - \frac y2
\bfv' + \frac12 \bfv. 
\]
Based on the assumptions \eqref{eq:Cond.bfA} and slightly generalizing the
results in \cite[Thm.\,3.1]{GalMie98DMSS}, we obtain that $\calA$ is a maximal
monotone operator which is strongly monotone, namely
\begin{equation}
  \label{eq:calA.monotone}
  \forall\,\bfv_1,\bfv_2\in \dom(\calA):\quad \langle
\calA(\bfv_1){-}\calA(\bfv_2),\bfv_1{-} \bfv_2 \rangle_\bfH \geq 
\int_\R \Big(\ALPHA|\bfv'_1{-}\bfv'_2|^2+ \frac12|\bfv_1{-}\bfv_2|^2 \Big)\dd y,
\end{equation}
and hence also coercive because of $\ALPHA>0$. Thus, \eqref{eq:IntegrStatEqn},
which now takes the form $\calA(\bfv)=\bfg$, has exactly one solution
$\bfv\in \bfH=\rmH^1(\R;\R^m)$ such that the unique solution
$\bfU=\wt\bfu_\pm+\bfv'$ is constructed.

For the reader's convenience and for checking that the vector-valued case works
exactly the same way, we repeat the argument. We first observe that
$\bfv \mapsto \calA_1(\bfv):=-\big( \bfA(\wt\bfu_\pm{+}\bfv')\big)'+\frac12
\bfv$ is monotone and continuous from $\bfH$ to $\bfH^*$, hence $\calA_1$ is a
maximally monotone operator, cf.\ \cite[Prop.\,32.7, p.\,854]{Zeid90NFA2b}.
Next, we consider the linear operator
$\bfv \mapsto \calA_2(\bfv):= \frac y2 \bfv'$ with
$\dom(\calA_2)=\dom(\calA)$. Hence, $\calA_2$ is maximally monotone by
\cite[Thm.\,32.L, p.\,897]{Zeid90NFA2b}. (A linear operator $L$ is maximally
monotone if and only if $L$ and $L^*$ are monotone and $L$ has a closed graph.)
With this we conclude that $\calA=\calA_1+\calA_2$ is maximally monotone by
\cite[Thm.\,32.I, p.\,897]{Zeid90NFA2b}, as both are maximally monotone and
$\dom(\calA_2)\cap \mafo{int}\big(\dom(\calA_1)\big) = \dom(\calA)\cap \bfH =
\dom(\calA) \neq \emptyset$.  Finally, we use \eqref{eq:calA.monotone} with
$\bfv_2=0$ and conclude that $\calA$ is strongly coercive, i.e.\
$\langle \calA(v),v\rangle/\| v\|_\bfH \to \infty $ for
$\| v\|_\bfH \to \infty$. Then, \cite[Cor.\,32.35, p.\,887]{Zeid90NFA2b}
implies that $\calA$ is surjective.\smallskip

\STEP{Step 3. A priori estimates for $\ALPHA>0$:} The first a priori estimate
is obtained by 
testing \eqref{eq:IntegrStatEqn} with $\bfv$ itself and using the monotonicity
of $\bfA$. Recalling $\Delta_\pm:= |\bfU_+{-}\bfU_-| $ and employing
\eqref{eq:Cond.bfA.b} we have  
\begin{align}
\nonumber
&\ALPHA\|\bfv'\|_{\rmL^2}^2{+}\frac34\|\bfv\|_{\rmL^2}^2  
\leq  \int_\R\Big( \big(\bfA(\wt\bfu_\pm{+}\bfv'){-}\bfA(\wt\bfu_\pm) \big) 
  \vdot \bfv' {+}\frac34|\bfv|^2 \Big) \dd y
\\
\nonumber
& =\int_\R\Big( \big(\frac y2\bfv'-\frac12\bfv +\bfg\big)\vdot \bfv
   -\bfA(\wt\bfu_\pm)\vdot \bfv' + \frac34|\bfv|^2\Big) \dd y 
=\int_\R \big(\bfg\vdot \bfv -\bfA(\wt\bfu_\pm)'\vdot \bfv \big)\dd y 
\\
\label{eq:APrioriL2}
& \leq 
\big(\|\bfg\|_{\rmL^2} + \|\rmD\bfA(\wt\bfu_\pm)\wt\bfu'_\pm \|_{\rmL^2}\big) \|
\bfv\|_{\rmL^2} \leq   C_\chi \BETA^{3/4} \Delta_\pm \|\bfv\|_{\rmL^2} ,   
\end{align}
where the last estimate used \eqref{eq:wtu.pri.BETA} and
\eqref{eq:bfg.BETA}. 

A second a priori estimate for $\ALPHA>0$ uses the monotonicity which implies
$\rmD \bfA(\bfu)\bfw\vdot \bfw \geq \ALPHA |\bfw|^2$ such that $\ALPHA>0$ gives
the invertibility of $\rmD\bfA(\bfu) \in \R^{m\ti m}$. Thus,
\eqref{eq:IntegrStatEqn} implies that $\bfU=\wt\bfu_\pm{+}\bfv'$ lies in
$\rmH^1_\mafo{loc} (\R;\R^m)$ and satisfies the strong profile equation
$\big( \rmD\bfA(\bfU)\bfU'\big)' + \frac y2 \bfU'=0$ in $\bfH^*$. In
particular, we can test this equation with $\bfU{-}\wt\bfu_\pm=\bfv'$ giving
\[
\int_\R \rmD\bfA(\bfU)\bfU' \cdot \bfU'\dd y = \int_\R\big(
\bfA(\bfU)'\cdot \wt\bfu'_\pm + \frac y2\bfU'\cdot(\bfU{-}\wt\bfu_\pm)\big) \dd y.
\] 
With this, \eqref{eq:Cond.bfA.b}, and suitable integrations by part  we obtain
\begin{align*}
&\ALPHA\|\bfU'\|_{\rmL^2}^2 + \frac14 \| \bfU{-}\wt\bfu_\pm\|_{\rmL^2}^2
\leq \int_\R\Big(\rmD \bfA(\bfU)\bfU'\cdot \bfU' + \frac14 |
\bfU{-}\wt\bfu_\pm|^2 \Big) \dd y 
\\
&= \int_\R\Big(  \bfA(\bfU)'\vdot \wt\bfu'_\pm  + \frac y2
\bfU'\vdot(\bfU{-}\wt\bfu_\pm) + \frac14 | \bfU{-}\wt\bfu_\pm|^2 \Big) \dd y  
\\
&= \int_\R\Big( \bfA(\wt\bfu_\pm)'\vdot \wt\bfu'_\pm  +\big(
\bfA(\wt\bfu_\pm){-}\bfA(\bfU)\big)\vdot \wt\bfu''_\pm  + \frac y2 \wt\bfu'_\pm  \vdot 
(\bfU{-}\wt\bfu_\pm)  \Big) \dd y
\\
& \leq \BETA \|\wt\bfu'_\pm \|_{\rmL^2}^2 + \BETA \| \bfU{-}\wt\bfu_\pm\|_{\rmL^2}
\|\wt\bfu''_\pm \|_{\rmL^2} +   \|\frac y2\wt\bfu'_\pm \|_{\rmL^2}\|
\bfU{-}\wt\bfu_\pm\|_{\rmL^2} 
\\
&
\leq   C_\chi\big( \BETA^{1/2} \Delta_\pm^2 +(\BETA\BETA^{-3/4} +
\BETA^{1/4})\Delta_\pm  \| \bfU{-}\wt\bfu_\pm\|_{\rmL^2} \big). 
\end{align*} 
Together with \eqref{eq:APrioriL2}  we have established the a priori estimate
\eqref{eq:APrioriEst}.\medskip

\STEP{Step 4. Exponential convergence for $\ALPHA,\delta>0$:} Using $\ALPHA>0$
we have shown that the unique solution $\bfU =\wt\bfu_\pm+\bfv'$ has the
regularity $\bfv\in \rmH^2(\R;\R^m)$. Thus, equation \eqref{eq:IntegrStatEqn}
shows that the flux
$\bfq : y \mapsto \bfA(\bfU(y))'=\rmD\bfA(\bfU(y))\bfU'(y))$ lies in
$\rmH^1_\mafo{loc}(\R;\R^m)$. Because of $\ALPHA>0$ the Jacobian
$\rmD\bfA(\bfu)\in \R^{m\ti m}$ is invertible, which shows that $\bfq$
satisfies $\bfq'+ \frac y2 \rmD\bfA(\bfU)^{-1}\bfq=0$. Thus, for $y\geq 0$ we
find
\[
\frac \rmd{\rmd y} \,| \bfq|^2 = - y \,\bfq \,\vdot\,
\rmD\bfA(\bfU)^{-1} \bfq \leq - y \delta \,| \bfq|^2,
\]
where we used \eqref{eq:Cond.bfA.c} with $\bfw=\rmD \bfA(\bfU)^{-1}
\bfq =  \bfU'$. Arguing similarly for $y\leq 0$ we arrive at 
$|\bfq(y)| \leq \ee^{-\delta y^2/4}\,|\bfq(0)|$ for $y \in \R$.
Evaluating \eqref{eq:IntegrStatEqn} at $y=0$ gives $\bfq(0)=\frac12 \bfv(0)-
\bfg(0)$. Using $\bfv(0)=  \int_{-\infty}^0 \ee^{y/\sqrt{\BETA}}\big( \bfv'(y)+
\bfv(y)/\sqrt{\BETA} \big) \dd y $ together with \eqref{eq:APrioriEst} and the
scaling properties of $\bfg$ yields  $|\bfq(0)| \leq
C_\chi \BETA^{1/2}\Delta_\pm $. Hence, the flux estimate 
\eqref{eq:FluxEst} is established. 

Having $\bfq$ under control, we return to the main equation
\eqref{eq:IntegrStatEqn} taking now the form $\bfq+\frac y2 \bfv' - \frac12
\bfv + \bfg=0$ and find the explicit representation in terms of $\bfq$:
\begin{equation}
  \label{eq:bfv.by.bfq}
  \bfv(y)=\begin{cases} 
   y \int_y^\infty \bfh(\eta) \dd \eta& \text{ for }y>0,\\
   2\bfg(0){+}2\bfq(0)       & \text{for } y=0,\\
   -y \int_{-\infty}^y\bfh(\eta) \dd \eta& \text{ for }y <0,
  \end{cases} \quad 
  \text{where } \bfh(\eta) = \frac2{\eta^2}\big(\bfg(\eta){+}\bfq(\eta)\big).
\end{equation}
By construction $\bfg$ has support in $[-\sqrt{\BETA}, \sqrt{\BETA}]$ and
satisfies $\|\bfg\|_{\rmL^\infty}\leq C_\chi \BETA^{1/2} \Delta_\pm$. Hence, we
obtain $|\bfg(y)| \leq C_\chi\BETA^{1/2} \Delta_\pm \ee^{-y^2/\BETA}
$. Setting $\gamma=\min\{1/\BETA,\delta/4\}>0$ and recalling \eqref{eq:FluxEst}
we find $|\bfh(y)|\leq C_\chi \BETA^{1/2} \Delta_\pm \ee^{-\gamma y^2 }/y^2$
and conclude
\[
|\bfv(y)| \leq  C_\chi \BETA^{1/2} \Delta_\pm\,\Phi\big(\sqrt{\gamma}\,y\big ) \quad
\text{with }\Phi(z) := z \int_z^\infty \frac{2\ee^{-r^2}} {r^2} \dd
r \leq 2 \ee^{-z^2}. 
\]  
Note that $\Phi$ has a continuous extension at $z=0$ with $\Phi(0)\leq 2$, 
such that we also have a uniform bound for $\bfv$ in the case $\delta=0$. 

With $y\bfv'(y)= \bfv (y) - 2\bfq(y)- 2\bfg(y)$ we obtain the pointwise a priori estimate
\begin{equation}
  \label{eq:Decay.bfv.bfv'}
  |\bfv(y)|+ |y\bfv'(y)| \leq C_\chi \BETA^{1/2} \Delta_\pm \ee^{-\gamma y^2} 
  \quad \text{for }y \in \R.\medskip
\end{equation}

\STEP{Step 5. The degenerate case with $\ALPHA=0$:} We study the auxiliary
problem where $\bfA$ is replaced by
$\bfA_\eps:\bfu \mapsto \bfA(\bfu)+ \eps \bfu$ for $\eps \in {]0,1[}$. Then,
$\bfA_\eps$ satisfies the assumptions \eqref{eq:Cond.bfA} with
$\BETA{}_\eps = \BETA + \eps$, $\ALPHA{}_\eps= \eps>0$, and
$\delta_\eps= \delta/(1{+}\delta\eps)$. To see the latter, we set
$B=\rmD\bfA(\bfu)$ and $B_\eps=B{+}\eps I$ and observe
\begin{align*}
\delta_\eps |B_\eps\bfw|^2 &\leq  \frac{\delta}{1{+}\delta\eps} 
\Big( |B\bfw|^2 {+} 2\eps |\bfw| |B\bfw| {+} \eps^2|\bfw|^2\Big) 
\leq \frac{\delta}{1{+}\delta\eps} \Big( (1{+}\delta \eps) |B\bfw|^2 {+}
(\eps^2{+}\frac{\eps^2}{\delta \eps} \big)|\bfw|^2\Big) 
\\
&=\delta |B\bfw|^2 + \eps|\bfw|^2
\  \overset{\text{\eqref{eq:Cond.bfA.c}}}\leq \
\bfw\cdot B\bfw + \eps|\bfw|^2  = \bfw \vdot B_\eps \bfw,
\end{align*}
which is the desired replacement of \eqref{eq:Cond.bfA.c} for $\eps>0$.

By the previous steps, there are unique solutions
$\bfU_\eps =\wt\bfu_\pm+\bfv'_\eps$, where 
\eqref{eq:APrioriEst} provides a uniform bound for
$\bfv_\eps$ in $\rmH^1(\R;\R^m)$.  
Hence, after extracting a subsequence (not relabeled) we may assume 
\begin{equation}
  \label{eq:veps.StrongCvg}
  \bfv_\eps \rightharpoonup \bfv_0 \text{ in } \bfH=\rmH^1(\R;\R^m) \quad
\text{and} \quad  \bfv_\eps \to \bfv \text{ in } \rmL^2(\R;\R^m) . 
\end{equation}
For the strong convergence, we employ the uniform decay estimate
\eqref{eq:Decay.bfv.bfv'}, where the decay factor $\gamma_\eps= \min\{1/\BETA{}_\eps,
\delta_\eps/4\}$ is uniformly bounded away from $0$.
  
By the global Lipschitz continuity of $\bfA$ we also have boundedness of
$\bfa_\eps= \bfA(\wt\bfu_\pm{+}\bfv'_\eps)- \bfb$ with $\bfb(y):=
\bfA(\wt\bfu_\pm(y))$ and may assume
\[
\bfa_\eps \rightharpoonup \bfa_0 \text{ in } \rmL^2(\R;\R^m). 
\]
Clearly, for $\eps>0$ the function $\bfv_\eps$ solves \eqref{eq:IntegrStatEqn}
if and only if 
\begin{equation}
  \label{eq:aeps.veps}
  0= \bfa'_\eps +\bfb' + \frac y2 \bfv'_\eps - \frac12\bfv_\eps + \bfg \quad 
  \text{in }\bfH^*=\rmH^{-1}(\R;\R^m). 
\end{equation}
Using the weak convergences of $\bfv_\eps$ in $\rmH^1$ and $\bfa_\eps$ in
$\rmL^2$, we see that this relation holds also for $\eps=0$. To show that
$\bfv_0$ solves \eqref{eq:IntegrStatEqn}, or equivalently that the profile
$\bfU=\wt\bfu_\pm+\bfv'_0$ is a solution of \eqref{eq:WeakStatProf}, 
it remains to show that $\bfa_0(y) = \bfA(\wt\bfu_\pm(y){+}\bfv'_0(y)) - \bfb(y)$
a.e.\ on $\R$. 

By the monotonicity of
$\calB:\rmL^2(\R;\R^m)\to \rmL^2(\R;\R^m);\ \bfw \mapsto \bfA(\wt\bfu_\pm{+}\bfw){-}
\bfb$ and Minty's monotonicity trick (see e.g.\ \cite[Ch.\,25(4),
p.\,474]{Zeid90NFA2b}), it suffices to shows that
$ \int_\R \bfa_\eps \vdot \bfv'_\eps \dd y \to \int_\R \bfa_0 \vdot
\bfv'_0 \dd y$ for ${\eps\to 0^+}$.  For this, we can exploit
\eqref{eq:aeps.veps} as follows:
\begin{align*}
\int_\R \!\!\bfa_\eps {\vdot} \bfv'_\eps \dd y
&= \int_\R \!\!{-}\bfa'_\eps \vdot \bfv_\eps \dd y
 \ \overset{\text{\eqref{eq:aeps.veps}}}= \ 
 \int_\R \!\!\big(\bfb' {+} \frac y2 \bfv'_\eps 
    {-} \frac12\bfv_\eps {+} \bfg  \big){\vdot} \bfv_\eps \dd y 
= \int_\R\!\!\big( \bfb'{\vdot} \bfv_\eps {-} \frac34|\bfv_\eps|^2 
   {+} \bfg {\vdot} \bfv_\eps\big) \dd y 
\\
&\to 
 \int_\R \!\!\big( \bfb'{\vdot} \bfv_0 {-}\frac34|\bfv_0|^2 {+} 
\bfg {\vdot} \bfv_0\big) \dd y 
 \ \overset{\text{\eqref{eq:aeps.veps}}}= \ 
-\int_\R \bfa'_0 \vdot \bfv_0 \dd y\ = \ \int_\R \bfa_0 \vdot \bfv'_0 \dd y,
\end{align*}
where ``$\to$'' uses the strong convergence \eqref{eq:veps.StrongCvg}.
Thus, Minty's trick gives $\bfa_0=\calB(\bfv'_0)=\bfA(\wt\bfu_\pm{+}\bfv'_0)
-\bfb$ and \eqref{eq:IntegrStatEqn} and \eqref{eq:WeakStatProf} are
established. 

The uniqueness of $\bfv_0$ again follows by strict monotonicity, see
\eqref{eq:calA.monotone} with $\ALPHA=0$. 

\STEP{Step 6. Two relations:} Using the fast decay of $\bfU-\wt\bfu_\pm$
arising from $\delta>0$ we can evaluate the indefinite integrals as follows:
\begin{align*}
0&= 2\bfA(\bfU(y))'\big|_{-\infty}^\infty= \int_\R 2(\bfA\circ \bfU)'' \dd y = -
\int_{-\infty}^0 y \;\!\bfU'\dd y + \int_0^\infty y \;\! \bfU'\dd y \\
&= \big[y(\bfU{-}\bfU_-)\big]_{-\infty}^0-\int_{-\infty}^0 \!(\bfU{-}\bfU_-)\dd y 
   + \big[y(\bfU{-}\bfU_+)\big]_0^\infty -\int_0^\infty \!(\bfU{-}\bfU_+)\dd y 
\\
& = -\int_\R \big( \bfU(y) - \ol\bfu_\pm(y)\big) \dd y,  
\end{align*}
which is the first relation in \eqref{eq:Rel.U.olUpm}.  Similarly, we obtain 
\begin{align*}
&\bfA(\bfU_+)-\bfA(\bfU_-)= \int_\R (\bfA\circ\bfU)' \dd y =  \\
&= \big[y(\bfA{\circ}\bfU)'\big]_{-\infty}^0\! 
   -\int_{-\infty}^0 \!y(\bfA{\circ}\bfU)''\dd y 
   + \big[y(\bfA{\circ}\bfU)'\big]_0^\infty 
     -\int_0^\infty \!\! y(\bfA{\circ}\bfU)''\dd y 
 = \int_\R\!\frac{y^2}2 \bfU'(y) \dd y 
\\
&= \big[\frac{y^2}2(\bfU{-}\bfU_-)\big]_{-\infty}^0 
     -\int_{-\infty}^0 \!y(\bfU{-}\bfU_-)\dd y 
   + \big[\frac{y^2}2(\bfU{-}\bfU_+)\big]_0^\infty 
      -\int_0^\infty \!y(\bfU{-}\bfU_+)\dd y 
\\
& = -\int_\R y \big( \bfU(y) - \ol\bfu_\pm(y)\big) \dd y,   
\end{align*}
which is the second relation in \eqref{eq:Rel.U.olUpm}.

\STEP{Step 7. Further regularity:} We know $\bfv \in \rmH^1:=\rmH^1(\R,\R^m)$, which
implies $\bfU\in \rmL^2_\mafo{loc}$. From the weak equation
\eqref{eq:WeakStatProf.b} we conclude that $\bfH:y \mapsto \bfA(\bfU(y))$
lies in $\rmH^1_\mafo{loc}$ by applying the Lemma of du Bois-Reymond. Thus, the
invertibility of $\rmD\bfA(\bfU)$ allows to apply the implicit function theorem
giving $\bfU \in \rmH^1_\mafo{loc}$. 

If $\bfA$ satisfies the further smoothness \eqref{eq:bfA.smooth}, then we
obtain higher regularity of $\bfU$ by the classical bootstrap argument applied
to the equation $\big(\rmD\bfA(\bfU)\bfU'\big)' = - \frac y2 \bfU'$.
\end{proof}

It is interesting to compare the approximation $\bfA_\eps(\bfu)=\bfA(\bfu)+\eps
\bfu$ in Step 5 of this proof with the linear approximation $\bbA_\eps =
\bbA+\eps \bbI$ in Example \ref{ex:LinVect}, where the solutions are given
explicitly. Hence one can see that the approximation is needed for smoothness
and exponential decay of the flux. 

While \eqref{eq:APrioriEst} provides an a priori estimate for $\bfU-\wt \bfu_\pm$
in $\rmL^2(\R;\R^m)$, we now show that in the case $\ALPHA>0$ one can also
obtain a uniform bound, which will be useful in Section \ref{su:OneReact3Spec}.  

\begin{corollary}[Uniform bound on $\bfU-\wt\bfu_\pm$] 
\label{co:UniformBdd} Assume the conditions \eqref{eq:Cond.bfA} with
$\delta, \ALPHA>0$, then the unique solution $\bfU:\R\to \R^m$ obtained in Theorem
\ref{th:VectValProf} satisfies
\begin{equation}
  \label{eq:UnifBdd.bfU}
  \big| \bfU(y) - \wt\bfu_\pm(y) \big| \leq C_\chi\, 
  \frac{\BETA^{1/2}}{\delta^{1/2} \ALPHA}\, \Delta_\pm \  \text{ for all } y
  \in \R, \quad \text{ where }   \Delta_\pm=|\bfU_+ {-} \bfU_-|. 
\end{equation}
\end{corollary}
\begin{proof} We set $\bbA(y)=\rmD \bfA(\bfU (y))$ and observe that
  \eqref{eq:Cond.bfA.b} implies $\langle \bbA \bfv, \bfv\rangle \geq \ALPHA
  |\bfv|^2$. Inserting $\bfv= \bbA^{-1} \bfw$ we obtain $|\bbA^{-1} \bfw| \leq
  |\bfw|/\ALPHA$. Now exploiting the flux estimate \eqref{eq:FluxEst} yields
\[
|\bfU'(y)| = |\rmD\bfA(U(y))^{-1} \bfq(y)| \leq \frac1{\ALPHA} C_\chi
\BETA^{1/2} \,\ee^{-\delta y^2/4}\, \Delta_\pm \quad \text{for all } y \in \R.  
\]
Using $\bfU(y)-\ol\bfu_\pm(y)= \int_{-\infty}^y \bfU'(z) \dd z$ for $y < 0$
and $\bfU(y)- \ol\bfu_\pm(y)= -\int_y^\infty \bfU'(z) \dd z$ for $y> 0$, we obtain
the desired estimate \eqref{eq:UnifBdd.bfU} if we take into account $|
\wt\bfu_\pm (y) - \ol\bfu_\pm(y)| \leq \Delta_\pm$. 
\end{proof}

We conclude this section on existence and uniqueness of similarity profiles $\bfU$
solving the weak form \eqref{eq:WeakStatProf} of the profile equation
\eqref{eq:I.ProfileEqn} with the important remark, that our result provides
existence also in the degenerate case with $\ALPHA=0$. While for $\ALPHA>0$ the
solutions are automatically smooth, see the statement after
\eqref{eq:bfA.smooth} in Theorem \ref{th:VectValProf}, the case allows for
discontinuous solutions $\bfU$ or for continuous solutions where $\bfU'$ has
singularities. The latter case will be important in the scalar situation
discussed in the following section.

\section{The case of scalar profiles}
\label{se:ScalProfiles}

We now restrict to the scalar case and consider the problem
\begin{equation}
  \label{eq:ScalSeSimEq}
  \big( \bbD(U)U'\big)' + \frac y2 \,U'=0 \quad \text{ for }y\in \R,\qquad 
U(\pm\infty) = U_\pm,
\end{equation}
where we always assume that $\bbD:\R\to \R$  is continuous and nonnegative. 
We observe that $A(u)=\int_0^u \bbD(s)\dd s$ is monotone and $\rmC^1$ with
$A'(u)=\bbD(u)$.  We have strict monotonicity if $\bbD$ has only isolated
zeros. The global conditions \eqref{eq:Cond.bfA} are satisfied with 
\[
\BETA =\sup_{u\in \R} \bbD(u), \quad \ALPHA= \inf_{u\in \R} \bbD(u), \quad 
\delta= 1/\BETA = \inf_{u\in \R} \frac1{\bbD(u)}. 
\] 
However, below we will show that all stationary profiles $U:\R\to \R$ are
monotone, e.g.\ for $U_-<U_+$ the profile is nondecreasing and satisfies
$U(y)\in [U_-,U_+]$ for all $y\in \R$. Hence, it will be sufficient to restrict
the above infima and supremum to the interval $[U_-,U_+]$.

Before giving the general existence theory, we look at a few examples with
degenerate diffusion, that is $\ALPHA=0$. 
 
\begin{example}
[Degenerate profiles]
\label{ex:DegenProfiles}\mbox{ } 

(I) As a first simple example we have 
\begin{equation}
  \label{eq:Exam.D.degen}
\bbD(u)=(1{-}u^2)/4 \quad \text{and} \quad u(y)= \begin{cases} 
U_-=-1 & \text{for }y\leq -1, \\ y &\text{for }y \in [-1,1],\\ 
U_+=1&\text{for }y\geq 1. \end{cases}
\end{equation}
(II) A second example can be constructed by setting 
\[
\wt u(y)= \begin{cases} U_-=-1& \text{for } y\leq -1,\\
\frac32 y - \frac12 y^3 &\text{for } y\in [-1,1],\\ 
U_+=+1 &\text{for } y\geq 1.  \end{cases}
\]
By exploiting equation \eqref{eq:ScalSeSimEq} for $y\in {]{-}1,1[}$ we 
obtain 
\[
\wt \bbD(u) = \frac18 \big( 1- \wt Y(u)^2\big), \quad \text{where } \wt Y:=
\big(\wt u|_{[-1,1]}\big)^{-1} :[-1,1]\to [-1,1].
\]
We observe that $\wt \bbD(u)=c |u{\mp}1|^{1/2}+ \mafo{h.o.t.}$ for $1 \mp u
\to 0^+$. 
\\[0.3em]
(III) As a third example we set $\wh Y:[-1,1]\to [-1,1]; u \mapsto \frac32 u -\frac12
u^3 $ and define 
\[
\wh u(y) = \begin{cases} U_-=-1& \text{for }y\leq -1,\\ 
\big(\wh Y|_{[-1,1]}\big)^{-1} (y)&\text{for }y\in [-1,1],\\ 
U_+=1& \text{for }y \geq 1. 
\end{cases}
\] 
A direct calculation shows that \eqref{eq:ScalSeSimEq} is satisfied for
$\bbD(u)= \frac3{16}(1{-}u^2)^2(5{-}u^2)$.
\\[0.3em]
(IV) More generally, using advanced ODE techniques, one can show that for
$\bbD$ satisfying $\bbD(u)=d_0\,(u{-}U_-)^\theta$ with $\theta>0$ that there
exists $y_-^*<0$ such that the profile $U$ satisfies $U(y)=U_- $ for
$y\leq y_-^*$ and $U(y)=c_0 \,(y{-}y_-^*)^{1/\theta}+ \mafo{h.o.t.}$ \ Indeed,
the three cases (I) to (III) above correspond to $\theta =1$, $1/2$, and $2$,
respectively. 

Moreover, for $\theta>1$ this shows that $U'$ can only lie in
$\rmL^p_\mafo{loc}(\R)$ for $p< \theta/(\theta{-}1)$, which is exactly the
restriction in Theorem \ref{th:SeSimiDegerate} below.
\\[0.3em]
(V) However, if $\bbD$ has an interior zero $U_0\in {]U_-,U_+[}$ of the form
$\bbD(u)=d_0 |u{-}U_0|^\theta+ \mafo{h.o.t.}$, then the profile $U$ will behave
like $U(y)=U_0 + c(y{-}y_0)^{1/(1{+}\theta)}+\mafo{h.o.t.}$ \ Thus, we find
$U'(y)\sim |y{-}y_0|^{-\theta/(1+\theta)}+ \mafo{ h.o.t.}$, which is a stronger
singularity than those that would occur near $U_\pm$.

The difference between the singularities at the boundaries $U_0\in \{U_-,U_+\}$
and in the interior $U_0 \in {]U_-,U_+[}$ is 
explained as follows: Interior singularities occur at positive continuous flux
$0< q(y_0)$ with $q(y)= \bbD(U(y)) U'(Y) \approx q(y_0)$, while at the two end
points singularities occur at flux $q(U_\pm)=0$.
\end{example}

To derive our subsequent a priori estimates we use the strategy implemented in
Step 5 of
the proof for Theorem \ref{th:VectValProf}. We will derive the estimates for
the case that $\bbD$ is bounded from below by $\ALPHA>0$ and then we conclude
that the same result holds for the case $\ALPHA=0$ by taking the limit for
$D_\eps = \eps + \bbD$. For $\eps \in [0,1]$ let $U_\eps$ denote the solution
of \eqref{eq:ScalSeSimEq} with $\bbD$ replaced by $D_\eps$, then we have 
\[
U_\eps {-} \wt u_* \,\rightharpoonup\, U_0 {-} \wt u_* \ \text{ in } \rmL^2(\R) 
\quad \text{and} \quad 
q_\eps =D_\eps (U_\eps) U'_\eps  \,\rightharpoonup \, q_0=\bbD(U_0)U'_0 \ 
\text{ in } \rmH^{-1}(\R). 
\]

Our first result concerns the monotonicity of the profiles, which then
will also improve the convergence by using Helly's selection principle for
sequences of monotone functions.

\begin{lemma}[Monotonicity of scalar profiles]
\label{le:ScalMonot} 
Assume that $(U_-,U_+)\in \R^2$ are given with $U_-< U_+$ and that
$\bbD \in \rmC^0(\R)$ is nonnegative. Then, the unique front $U$ provided
by Theorem \ref{th:VectValProf} is nondecreasing and hence only depends
on $\bbD|_{[U_-,U_+]}$.
\end{lemma}
\begin{proof}
We first consider the case $\ALPHA=\min\bigset{\bbD(u)}{ u\in \R} >0$, then the
flux $q(y) = \bbD(U(y))U'(y)$ satisfies the ODE $q' + y q/(2\bbD(u))=0$. Hence,
$q$ cannot change its sign. Moreover, the boundary conditions impose 
$\int_\R q \dd y = \int_\R \bbD(U)U' \dd  y = \int_{U_-}^{U_+} \bbD(u) \dd
u\geq \alpha (U_+{-}U_-)>0$. Hence, we find $q\geq 0$ and thus $U'(y) \geq 0$. 

The general case with $\ALPHA=0$ follows by approximation.
\end{proof}

Restricting to the case $U_-< U_+$ we now define the relevant constants 
\[
D_*:=\min\bigset{\bbD(u)}{ u\in [U_-,U_+] } 
\quad \text{and} \quad 
D^*:=\max\bigset{\bbD(u)}{ u\in [U_-,U_+] },
\]
which satisfy $\ALPHA \leq D_* < D^* \leq \BETA$. Clearly, all estimates on the
profile $U$ will only depend on $D_*$ and $D^*$.  Moreover, when approximating
$U$ by $U_\eps$ as indicated above, we can use the monotonicity $U'_\eps(y)\geq
0$ and employ Helly's selection principle to conclude 
\[
 U_\eps (y) \to U(y) \ \text{ at all continuity points $y \in \R $ of } U.
\]

We proceed by supplying further a priori estimates that are
essentially contained in
\cite[Thm.\,5]{VanPel77CSSN} and have their origin in
\cite{Sham76CDD3}. However, here we provide a much shorter direct proof.

\begin{proposition}[A priori estimates]\label{pr:Estim.unif.eps}
  Assume $U_-< U_+$ and that $\bbD \in \rmC^0([U_-,U_+])$ with
  $D_*=\min \bbD \geq 0$ and $D^*=\max \bbD$. Then, the solution $U$ and its
  associated flux $Q(y)=\bbD(U(y))U'(y)$ satisfy the estimates
\begin{equation}
  \label{eq:ueps.phieps}
  \begin{aligned}
   0& \leq U_+- U(y) \leq \big(U_+ - U(z)\big) \,\ee^{-(y^2-z^2)/(4 D^*)}
   &&\text{for }0\leq z \leq y, 
  \\
  0&\leq  U(y)- U_- \leq  \big(U(z)-U_-\big) \,\ee^{-(y^2-z^2)/(4 D^*)}
   &&\text{for } 0 \geq z\geq y, 
  \\
   0&\leq Q(\pm y) \leq Q(\pm z)  \,\ee^{-(y^2-z^2)/(4 D^*)}
   &&\text{for }0\leq z\leq y.
  \end{aligned}
\end{equation}
Moreover, the values $U(0) \in {]U_-,U_+[}$  and
$Q(0)>0$ are restricted by the inequalities 
\begin{equation}
  \label{eq:Bounds.unif.eps}
  \begin{aligned}
   \int_{U_-}^{U(0)} (s{-}U_-)\bbD(s) \dd s \leq 2 Q(0)^2 \leq
   (U(0){-}U_-) \int_{U_-}^{U(0)} \bbD(s) \dd s, 
  \\  
   \int_{U(0)}^{U_+} (U_+{-}s) \bbD(s) \dd s \leq 2 Q(0)^2 \leq
   (U_+{-}U(0))\int_{U(0)}^{U_+}  \bbD(s) \dd s.
  \end{aligned}
\end{equation}
\end{proposition}
\begin{proof} Again we may use $D_*\geq \ALPHA>0$ for deriving the following estimates.

We first show the last estimate in \eqref{eq:ueps.phieps}. For this we observe
\[
Q'(y) = - \frac y2 U'(y) =  \frac{-y}{2\bbD(U(y))}\; Q(y),
\] 
which gives the relation $Q(y)=\Psi_z(y) Q(z)$ with
$\Psi_z(y)=\exp\big(\int_z^y \frac{-\eta}{2\bbD(U(\eta))} \dd\eta \big) $.
Using $\bbD(u)\leq D^*$ gives the third estimate in \eqref{eq:ueps.phieps}. 

From this bound for $Q=\bbD(U)U'$ and the lower bound $\bbD \geq D_* >0$ we now
deduce that $U$ converges faster than exponential to its limits $U_\pm$ for
$y \to \pm \infty$. For $y\geq 0$ we set $w(y)= U_+ -U(y)$, and by integrating
\eqref{eq:ScalSeSimEq} over $y \in {]z,\infty[}$ we obtain
\[
\bbD(U(z))w'(z)= -\frac z2 w(z) - \beta(z) \quad \text{with }
\beta(z)=\frac12 \int_z^\infty w(y) \dd y \geq 0.  
\]
For $0\leq z\leq y$ Duhamel's formula gives 
\begin{align*}
U_+-U(y)&\  = \ w(y)= \Psi_z(y) w(z)- \int_z^y \Psi_\eta(y)
   \frac{\beta(\eta)}{\bbD(U(\eta))} \dd \eta
\\ 
&\overset{\beta\geq 0}\leq \; \Psi_z(y) w(z) \leq \ee^{-(y^2-z^2)/(4D^*)}
(U_+{-}U(0)).
\end{align*}
Together with the analogous result for $y\leq 0$ the estimates in
\eqref{eq:ueps.phieps} are established. 

For \eqref{eq:Bounds.unif.eps} it suffices to show the second line by
integration over $y\geq 0$. The first line follows similarly by integration
over $y\leq 0$. For the upper estimate we proceed as follows:
\begin{align*}
4Q(0)^2 & = \Big( \int_0^\infty 2Q'(y)\dd y \Big)^2
   \overset{\text{\eqref{eq:ScalSeSimEq}}}=  
   \Big( \int_0^\infty  y\,U'(y)\dd y \Big)^2 
\\
&\overset{\text{CS}}\leq  \int_0^\infty U' (y) \dd y \; 
   \int_0^\infty y^2 U'(y) \dd y \ = \ \big(U_+{-}U(0)\big)
   \int_0^\infty\!\! {-}2 y Q' (y) \dd y 
\\
&= \big(U_+{-}U(0) \big) \bigg[{-}2y Q(y)\Big|_0^\infty + \int_0^\infty
  2 Q(y) \dd y  \bigg] \ = \ \big(U_+{-}U(0)\big) \int_0^\infty\!\! 
  2 \bbD(U(y)) U'(y) \dd y 
\\
&=2 \big(U_+{-}U(0)\big) \int_{U(0)}^{U_+}  \bbD(s) \dd s,
\end{align*}
where we used the monotonicity $U'(y)\geq 0$. This is the
desired upper estimate for $Q(0)^2$. 

For the lower estimate we proceed as follows:
\begin{align*}
&\int_{U(0)}^{U_+} (U_+{-}s) \bbD(s) \dd s  = \int_{U(0)}^{U_+}
\int_{U(0)}^u \bbD(s) \dd s \dd u = \int_{y=0}^\infty \int_{U(0)}^{U(y)}
\bbD(s) \dd s  \:U'(y) \dd y
\\ &
= \int_{y=0}^\infty \int_{z=0}^y  \bbD(U(z)) U'(z) \dd z \:U'(y) \dd y = 
\int_{y=0}^\infty \int_{z=0}^y  Q(z) \dd z  \:U'(y) \dd y
\\
&\overset{**}\leq  \int_{y=0}^\infty y\,Q(0)\:U'(y)\dd y
\overset{\text{\eqref{eq:ScalSeSimEq}}}=    Q(0) \int_0^\infty \!\!{-}2
Q'(y) \dd y \ = \ 2 Q(0)^2,
\end{align*}
where in $\overset{**}\leq $ we used $Q(z)\leq Q(0)$
from \eqref{eq:ueps.phieps}.  
Hence, Proposition \ref{pr:Estim.unif.eps} is established.
\end{proof}

As a simple consequence of \eqref{eq:Bounds.unif.eps} we see that $Q(0)> 0$
and $U(0) \in {]U_-,U_+[}$ as soon as $\bbD \in \rmC^0([U_-,U_+])$ is
nontrivial. In the case of constant $\bbD$ we have
$D_*=D^*$ and the explicit linear solution gives $U(0)=\frac12(U_-+U_+)$ and $Q(0)=
 \sqrt{D_*/(4\pi)} \, (U_+{-}U_-) $.  
Moreover, we obtain upper and lower bounds for $Q(0)$ and $U(0)$
in the case $D_*>0$. These results are valuable if $D_*/D^*$ is close to $1$
but deteriorate for $D_*/D^* \approx 0$. 

\begin{corollary}[Simple bounds on $Q(0)$ and $U(0)$]
\label{co:BoundsQ0U0}
Assume $D^*\geq \bbD(u)\geq D_*>0$ for all $u\in [U_-,U_+]$ and set $\gamma=
\sqrt{D_*/(2D^*)} \leq \sqrt{1/2}$. Then, the unique profile $U$ with
$U(\pm\infty)=U_\pm$ satisfies 
\[
U(0) \in \big[\frac{U_-+\gamma U_+}{1+\gamma} , \frac{\gamma U_- +
  U_+}{1+\gamma} \big] \quad \text{and} \quad 
Q(0) \in \big[ \sqrt{D_*/16}\, (U_+{-}U_-), \sqrt{D^*/8} \, (U_+{-}U_-)\big] .
\]   
In particular, we have $0\leq U'(y)  \leq (U_+{-}U_-) \sqrt{D^*/(8D_*^2)\,}$. 
\end{corollary} 
\begin{proof} 
We simply insert the upper and lower bound for $\bbD$ into
\eqref{eq:Bounds.unif.eps} and find
\begin{align*}
&  \frac12(U(0){-}U_-)^2 D_* \leq 2 Q(0)^2 \leq (U(0){-}U_-)^2 D^*
 \quad \text{and} \\
& \frac12(U_+{-}U(0))^2 D_* \leq 2 Q(0)^2 \leq (U_+{-}U(0))^2 D^*. 
\end{align*}
From this the first two estimates follow easily. 

The derivative satisfies $U'(y) = Q(y)/\bbD(U(y)) \leq Q(0)/D_*$ giving 
the result. 
\end{proof}

With this information we can pass to the limit and obtain the following
existence result. The conditions for deriving $U'\in \rmL^p(\R)$ are indeed
sharp (but leaving the critical cases open), as can be seen by comparing with
the cases (IV) and (V) in Example \ref{ex:DegenProfiles}. 

\begin{theorem}
[Self-similar fronts in the degenerate case]
\label{th:SeSimiDegerate}
Assume $U_-< U_+$ and that $\bbD\in \rmC^0([U_-,U_+]) $ satisfies
\begin{equation}
  \label{eq:Cond.D.theta}
\exists \, \theta \in {]0,1[}, \ p\in {[1,\infty[}: \quad 
\wt C_{p,\theta} := \int_{U_-}^{U_+} \! \Big( 
\frac{(U_+{-}u)^\theta(u{-}U_-)^\theta}{\bbD(u)}\Big)^{p-1} \dd u < \infty.
\end{equation}
Then, the unique and monotone solution $U$ of \eqref{eq:ScalSeSimEq} satisfies 
$U' \in \rmL^p(\R)$, namely 
\begin{align}
  \label{eq:Lp.U'}
\| U'\|_{\rmL^p(\R)}^p \leq \wh C_\theta^{p-1} \wt C_{p,\theta} \quad \text{with } 
 \wh C_\theta:= \sqrt{\frac{2D^*}{1{-}\theta}} \:
 \frac{U_+ - U_- }{  (U_+{-}U(0))^\theta (U(0){-}U_-)^\theta  }.
\end{align}

Moreover, if $\int_{U(0)}^{U_+} \bbD(s)/(U_+{-}s)\dd s <\infty$, then we
have 
\begin{equation}
  \label{eq:Constancy}
  U(y)=U_+ \ \text{ for } y \geq y_+^*:= \frac{U_+ {-} U(0)}{Q(0)}
  \,\int_{U(0)}^{U_+} \frac{\bbD(u)}{U_+{-}u} \dd u >0,
\end{equation}
and an analogous statement holds for $y\leq y_-^* := -
 \frac{U(0) {-} U_-} {Q(0)}
  \,\int_{U_-}^{U(0)} \frac{\bbD(u)}{u{-}U_-} \dd u <0 $. 
\end{theorem}
\begin{proof} We again assume $\bbD(U)\geq D_*>0$ such that $U$ is smooth. 

\STEP{Step 1. Bound for $Q$ in terms of $\min\{U{-}U_-,U_+{-}U)\}$:} We first
show that for all $\theta \in {[0,1[}$ there exists $C_\theta$ such that 
\begin{equation}
  \label{eq:QvsU.theta}
\forall\, y\in \R:\quad   Q(y) \leq \sqrt{\frac{2D^*}{1{-}\theta}} 
\Big(\frac{(U_+{-}U(y))(U(y){-}U_-)}{(U_+{-}U(0))(U(0){-}U_-) } \Big)^\theta 
\big(U_+ - U_-\big).
\end{equation}
It is sufficient to estimate $Q(y)$ by $(U_+{-}U(y))^\theta$  for $y\geq
0$ and by $(U(y){-}U_-)^\theta$ for $y\leq 0$. We concentrate on $y>0$, the case
$y<0$ is similar. From $2Q'=- y U'$ we obtain
\[
2Q(y)= -\int_y^\infty 2Q'(z) \dd z =\int_y^\infty z (U'(z){-}U_+)\dd z =
y(U_+ {-} U(y)) + \int_y^\infty\!\big( U_+ {-} U(z) \big) \dd z.
\] 
Using the first estimate in \eqref{eq:ueps.phieps}, the last term can be
estimated via
\begin{align*}
\int_y^\infty\!\big( U_+ {-} U(z) \big) \dd z &\leq (U_+{-}U(y)) \int_y^\infty
\ee^{-(z^2-y^2)/(4D^*)}\dd z\\
&  \leq  (U_+{-}U(y)) \int_y^\infty
\ee^{-(z-y)^2/(4D^*)}\dd z = (U_+{-}U(y)) \sqrt{D^*/\pi}. 
\end{align*}
Applying the first estimate in \eqref{eq:ueps.phieps} once again we find
\begin{align*}
Q(y) & \leq \frac12\big( y + \sqrt{\pi D^*}\,\big) (U_+{-}U(y)) \leq
\sqrt{\frac{2D^*}{1{-}\theta}}  \,
  (U_+{-}U(y))^\theta (U_+{-}U(0))^{1-\theta}, 
\end{align*}
where for $y\geq 0$ we estimated $\frac12 \big( y + \sqrt{\pi D^*}\,\big)\, 
  \:\ee^{-(1{-}\theta) y^2/(4D^*)} $  $\leq$ $
  \big((2\ee(1{-}\theta))^{-1/2} +  \sqrt{\pi/4}  \big) 
      \,\sqrt{D^*}$ $ \leq $ $\sqrt{2D^*/(1{-}\theta)}$. 
Moreover, monotonicity gives $U(y) \in [U(0),U_+]$ and we conclude
\[
(U_+{-}U(y))^\theta (U_+{-}U(0))^{1-\theta}  \leq 
\Big(\frac{(U_+{-}U(y))(U(y){-}U_-)}{(U_+{-}U(0))(U(0){-}U_-) } \Big)^\theta 
\big(U_+ - U_-\big).
\]
Thus, \eqref{eq:QvsU.theta} is shown for $y\geq 0$ and the result for $y\leq 0$
follows analogously.

\STEP{Step 2. $\rmL^p$ estimate for $U'$:}  We abbreviate
$\delta(y)=\bbD(U(y))$ and $\mu(u)=( U_+{-}u) (u{-}U_-)$. 
Recalling $Q=\delta U'$ and writing estimate
\eqref{eq:QvsU.theta} as $\delta U'\leq C_* \mu(U(y))^\theta$ we  obtain
\begin{align*}
  \int_\R (U')^p \dd y &= \int_\R \big(
  \frac{\delta(y)U'(y)}{\mu(U(y))^\theta}\big)^{p-1}\;
  \big(\frac{\mu(U(y))^\theta}{\delta(y)}\big)^{p-1} U'(y) \dd y 
\leq C_*^{p-1} \int_{U_-}^{U_+}  \big(\frac{\mu(u)^\theta } {\bbD(u) } 
\big)^{p-1} \dd u  < \infty,
\end{align*}
which is the desired estimate \eqref{eq:Lp.U'}. 

\STEP{Step 3: Constant values for $y\geq y_+^*$.} 
To show \eqref{eq:Constancy} consider $y>0$ such that $U'(z)>0$ for $z
\in [0,y]$. For this, note that $U'$ is continuous on the set ${]Y_-,Y_+[}$
which is defined by the condition $U(y)\in {]U_-,U_+[}$. On ${[0,Y_+[}$ we define the
auxiliary functions  
\[
w(y) = U_+-U(y)>0 \quad \text{and} \quad 
h(y)= \frac{2\bbD(U(y))U'(y)}{U_+- U(y)}= \frac{2Q(y)}{w(y)}= y + \frac1{w(y)}
\int_y^\infty w(z) \dd z,
\]
where the last identity results from integrating \eqref{eq:ScalSeSimEq} over
$z \in {[y,\infty[}$. We easily find
$h'(y)= -w'(y)\int_y^\infty w\dd z /w(y)^2 >0$ because of $U'=-w'>0$ and
conclude $h(y) \geq h(0)$ for all $y\in {[0,Y_+[}$.  With this and $U'>0$ we
obtain
\begin{align*}
  y& \ts =\int_0^y \dd z = \int_0^y \frac 2{h(z)}\:\frac{\bbD(U(z))}{U_+{-}U(z)}
  \: U'(z) \dd z
 \leq \frac2{h(0)} \int_0^y \frac{\bbD(U(z))}{U_+{-}U(z)}
  \: U'(z) \dd z
= \frac2{h(0)} \int_{U(0)}^{U(y)} \frac{\bbD(u)}{U_+{-}u} \dd u. 
\end{align*}
In the limit $y\to Y_+-0$ we find $U(y)\to U_+$ and conclude $Y_+\leq y_+^*$
after inserting $h(0)= 2 Q(0)/w(0)$.  The estimate $y_-^* \leq Y_-$ is shown
analogously.
\end{proof}

\section{Stability of profiles in the scalar case}
\label{se:PME}

The porous medium equation (see e.g.\ \cite{Vazq07PMEM}) is given by $u_t= \Delta
A(u)$, where one is typically interested in nonnegative solutions and $A$ is
defined only for $u\geq 0$. The classical choice, which we will also consider below,
is given by $A(u)=u^m$ for $m>0$. 

The one-dimensional case in parabolic scaling is given in the form 
\begin{equation}
  \label{eq:PDE1D}
  u_\tau = \big(A(u)\big)_{yy} + \frac y2 u_y=  \big(\bbD(u)u_y\big)_y + \frac
  y2 u_y, \qquad u(\tau,\pm \infty)=U_\pm .
\end{equation}

Applying the
theory of Section \ref{se:ScalProfiles} with $\bbD(u)=m u^{m-1}$, we can treat
the case $m\geq 1$ and obtain for all $0\leq U_- \leq U_+ < \infty$ a unique
monotone profile $U:\R\to [U_-,U_+]$ satisfying \eqref{eq:ScalSeSimEq}. For
$m=1$ we obtain the trivial solution \eqref{eq:ExplicLinCase} in terms of the 
error function. For $m>1$ there are two cases, namely (i) $U_->0$, which implies
that $U\in \rmC^\infty(\R; [U_-,U_+])$, and   (ii) $U_-=0$. In the latter case
we have $\int_0^u \bbD(s)/s \dd s = m \int_0^u s^{m-2} \dd s = \frac m{m-1}
u^{m-1} <\infty$ which implies $U(y)=0$ for all $y\leq y^*_{-}=Y(m,U_+)<0$, see
\eqref{eq:Constancy}. It
can be shown that $U \in \rmC^{1/(m-1)}(\R; [0,U_+])$ for $m>1$ and $m\neq
1+1/k$ for $k\in \N$; for $m=1+1/k$ one obtains $U\in
\rmC^{k-1,\mafo{lip}}(\R)$. Note that \eqref{eq:Lp.U'} in Theorem
\ref{th:SeSimiDegerate} 
implies $U' \in \rmL^\infty(\R)$ for $m\in {]1,2[}$ and $U' \in \rmL^
{(m-1)/(m-2)}(\R)$ for $m>2$. \medskip

Having the self-similar profile $U$ satisfying $A(U)_{yy}+ \frac y2 U_y=0$ and
$U(\pm\infty)=U_\pm>0$ we can also establish convergence of general solutions $u$
of \eqref{eq:PDE1D}, at least in some cases. 
We also refer to \cite{VanPel77ABSN} for convergence results to self-similar
profiles, but they are quite different and rely on comparison principle
arguments, whereas we use entropy estimates that may also be extended to
vector-valued cases, see \cite{MiHaMa15UDER,MieMit18CEER,MieSch22?CSSP}.  
For this
we introduce the relative entropy 
\[
\calH_\phi(u)=\int_\R \phi(u(y)/U(y)) U(y) \dd y, \quad \text{where }
\phi''(\rho)>0, \ \phi(1)=\phi'(1)=0.
\]
Typical entropy functions are given by the family $E_p:{[0,\infty[}\to
[0,\infty]$ via 
\begin{align*}
&E_p(\rho)=\frac1{ (p {-}1) p }\big(
\rho^p-p \rho +p-1\big) \text{ \ for } p\in\R\setminus\{0,1\}, \\
& E_1(\rho) :=\rho \log\rho -\rho +1, \quad E_0(\rho) = -\log \rho +\rho-1,
\end{align*}
which is uniquely determined by the conditions $E''_p(\rho)=\rho^{p-2}$ for $\rho >0$ 
and $E_p(1)=E'_p(1)=0$. Our entropy functions will be of the form 
\begin{equation}
\label{eq:phi.Epq}
\varphi_{p,q}(\rho)= \begin{cases} E_p(\rho) & \text{for }\rho \in [0,1], \\
E_q(\rho)& \text{for } \rho\geq 1, 
\end{cases}
\end{equation}
with suitable $p$ and $q$. 

Because of the multiplicative ansatz $u=\rho U$ and $\phi(\rho)>0$ for $\rho \neq 1$
in the definition of
$\calH$, the condition $\calH_\phi(u)< \infty$ implies that $u$ has to approach the
same limits as $U$. Moreover, $\calH_\phi(u(\tau))\to 0$ for $\tau\to \infty$
implies $u(\tau)\to U$ in a suitable sense, see below.  

A direct calculation, using the shorthand $\rho=u/U$, gives 
\begin{align*}
\frac\rmd{\rmd \tau} \calH_\phi(u(\tau)) &=\int_\R \phi'(\rho) \rho_\tau U \dd y
=  \int_\R \phi'(\rho)\big( A(\rho U)_{yy} + \frac y2(\rho U)_y\big) \dd y 
\\
&= \int_\R\Big({-}\phi'(\rho)_yA(\rho U)_y + \rho_y\phi'(\rho) \frac y2U
    +  \rho \phi'(\rho) \frac y2U_y\Big) \dd y 
\\
&\overset{*}= \int_\R \Big( {-}\phi''(\rho) \rho_y A'(\rho U)(\rho U)_y
-\phi(\rho)\frac12 U + \big(\phi(\rho) {-}\rho \phi'(\rho)\big) A(U)_{yy}
\Big) \dd y  , 
\end{align*}
where we have used the profile equation to substitute $\frac y2 U_y$ by
$-A(U)_{yy}$. Integrating the last term by parts, we arrive at the identity 
\begin{align*}
\frac\rmd{\rmd \tau} \calH_\phi(u(\tau)) &=-\frac12 \calH_\phi(u) -
\int_\R \phi''(\rho)\Big( A'(\rho U)U \rho_y^2  + \rho \big(A'(\rho
U){-}A'(U)\big) \rho_y U_y\Big) \dd y .
\end{align*}
We can estimate the last term by minimizing the integrand
with respect to $\rho_y$ pointwise  and obtain 
\begin{align}
  \label{eq:calH.estim}
\frac\rmd{\rmd \tau} \calH_\phi(u(\tau)) & \leq -\frac12 \calH_\phi(u) +
\int_\R \phi''(\rho)\frac{\rho^2 U_y^2\big(A'(\rho U){-} A'(U)\big)^2}{4
  A'(\rho U) U}  \dd y .
\end{align}
In the linear case $A(u) = \delta u$ the last term vanishes and 
$\calH_\phi(u(\tau))\leq \ee^{-\tau/2}\calH_\phi(u(0))$
follows immediately. For general $A$, one can prove exponential convergence if it is
 possible to estimate the
last term by $\varkappa \calH_\phi(u)=\varkappa  \int_\R \phi(\rho) U \dd y$ for some
$\varkappa < 1/2$. 

Before the consider the important special case $A(u)=u^m$ further down below,
we consider the case where $u \mapsto A'(u)$ is globally Lipschitz continuous,
namely 
\begin{equation}
  \label{eq:A.bdd.Lip}
  \exists\, \ALPHA>0, C_\rmA>0\ \forall \, u,v\geq 0: \quad 
A'(u)\geq \ALPHA \quad  \text{ and } \quad \big| A'(u)-A'(v)\big| \leq C_\rmA |u{-}v|,
\end{equation}
and derive an exponential decay estimate.  Under a suitable flatness
condition on  $U$ we obtain a uniform exponential decay on $\calH_\phi$ for all
initial conditions. 

\begin{theorem}[Exponential decay if $|A''(u)|\leq C_\rmA$]
\label{th:ExpDecay.CA}
Consider the diffusion equation \eqref{eq:PDE1D} with general $A\in\rmC^2(\R)$
satisfying assumption \eqref{eq:A.bdd.Lip} and choose $\phi=\varphi_{1/2,-1}$. 
Assume further that the
stationary profile $U \in \rmC^1(\R;[U_-,U_+])$ from \eqref{eq:ScalSeSimEq} satisfies 
\[
U_+\geq U_->0 \quad \text{and} \quad 
\Sigma_{U,0}:=\sup\bigset{U'(y)^2 }{y\in \R} < \ALPHA/C_\rmA^2 . 
\] 
Then, all solutions $u$ of \eqref{eq:PDE1D} with $u(0,y)\geq 0$ and 
$\calH_\phi(u(0)) < \infty$ satisfy the
exponential decay estimates  
\begin{align*}
&\calH_\phi(u(\tau)) := \int_\R \varphi_{1/2,-1}\big(u(\tau,y)/U(y)\big) \,U(y) \dd y \leq
\ee^{-\Lambda}\,\calH_\phi(u(0)) \quad \text{for all } \tau >0,  
\\
&\big\|\sqrt{ u(\tau,\cdot)} - \sqrt{U(\cdot)} \big\|^2_{\rmL^2(\R)} \leq
 \ee^{-\Lambda \tau }\, \calH_{\phi}(u(0)) \text{ \ for all
 } \tau>0,
\end{align*}
where $\Lambda = \dfrac12\big( 1- C_\rmA^2\, \Sigma_{U,0}/\ALPHA \big)$. 
\end{theorem} 
\begin{proof} 
We first observe that the choice $\phi=\varphi_{1/2,-1}$ leads to the estimate
\[
\phi''(\rho) \rho^2(\rho{-}1)^2 \leq 2 \phi(\rho)  \quad \text{for all }
\rho\geq 0,
\]
see Figure \ref{fig:EntropyCA}.  Using this
and \eqref{eq:A.bdd.Lip}, the estimate \eqref{eq:calH.estim} takes the form
\[
\frac{\rmd}{\rmd \tau}\calH_\phi(u) = - \frac12 \calH_\phi(u) + \int_\R 2\phi(\rho)
\,  \frac{U'(y)^2 \, C_\rmA^2 U(y)^2}{4 \, \ALPHA \, U(y)}  \dd y 
\leq   - \frac12 \big(1- \frac{C_\rmA^2}{\ALPHA}\, \Sigma_{U,0}\big)\,\calH_\phi(u). 
\]
This proves the first decay estimate, and the second follows by using 
\[
(\sqrt{u}-\sqrt{U})^2 = (\sqrt\rho - 1)^2\,U= \frac12 E_{1/2}(\rho)\, U \leq
\varphi_{1/2,-1}(\rho)\, U.  
\] 
Integration over $y\in \R$ completes the proof.
\end{proof} 
\begin{figure}
\begin{tikzpicture}
\node (p) at (0,0) {\includegraphics[width=0.4\textwidth]{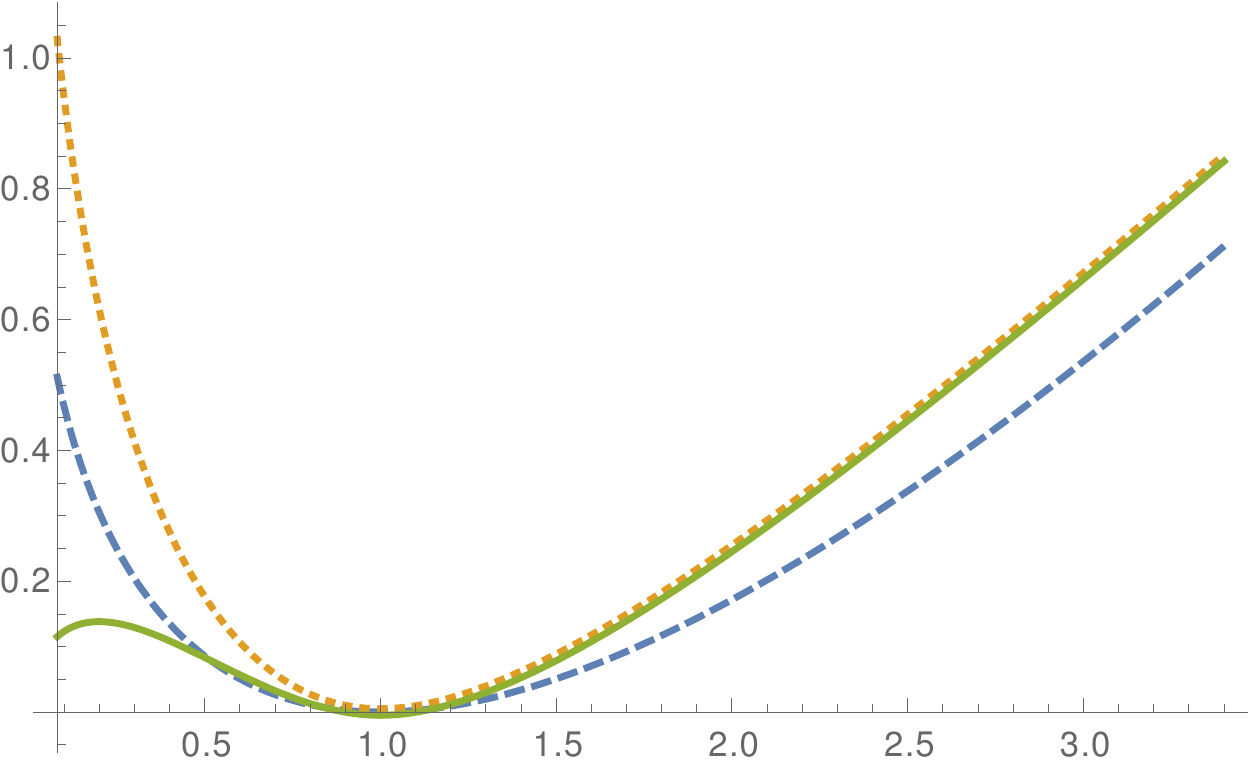}}; 
\node[color=orange!85!green] (p1) at (-1.8,1.6) {$\varphi_{1/2,-1}(\rho)$};
\node[color=teal!60!blue] (p2) at (2.4,-0.9) {$\tfrac12 E_{1/2}(\rho)$};
\node[color=olive!70!green] (p3) at (-3.7,-1.4) {$\psi(\rho)$};
\end{tikzpicture}\qquad 
\begin{minipage}[b]{0.35\textwidth}
\caption{The function $\psi(\rho)=\frac12\phi''(\rho)\rho^2(\rho{-}1)^2$ (full
  line, green) is 
  less or equal to $\phi=\varphi_{1/2,-1}$ (dotted orange). Moreover,
  $\varphi_{1/2,-1}$ lies above $\frac12 E_{1/2}(\rho)=(\sqrt\rho-1)^2$ (dashed, blue). 
  \vspace{0.4cm}  \mbox{} } 
\label{fig:EntropyCA}
\end{minipage}
\end{figure}
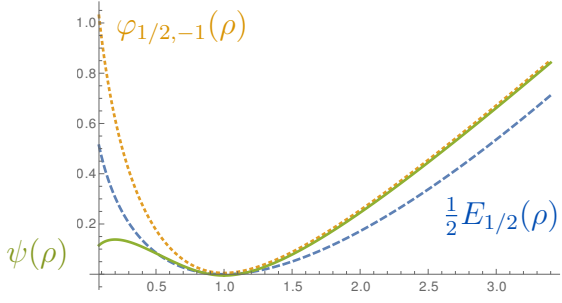

As a second example we restrict to the case $A(u)=u^m$, which leads to a strong
simplification because the integrand in the last term in \eqref{eq:calH.estim}
can be factored in the form $\phi''(\rho) B_m(\rho) U_y^2 U^{m-2}$ for some
$B_m$.  Proceeding as for the last result we find the following decay
estimates.

\begin{theorem}[Convergence in the PME $A(u)=u^m$]
\label{th:PME.Cvg}
Consider the porous medium equation \eqref{eq:PDE1D} with $A(u)=u^m$ for $m\geq
1$ and let $U \in \rmC^0(\R;[U_-,U_+])$ denote
the similarity profile satisfying \eqref{eq:ScalSeSimEq} and
\begin{equation}
  \label{eq:PME.Sigma.Cond}
  U_+\geq U_->0 \quad \text{and} \quad 
\Sigma_{U,m}= \sup\bigset{U'(y)^2 U(y)^{m-2}}{ y \in \R} < \frac1{m(m{-}1)^2}.
\end{equation}
Then, choosing the entropy density function 
\[
\phi_m:=\varphi_{p_m,q_m} \quad \text{with } p_m= \max\{1/2, m{-}1\} \text{ and }
q_m=\min\{ 1/2, 2{-}m\},
\]
all solutions of \eqref{eq:PDE1D} with $u(0,y)\geq 0$ satisfy the global decay
estimates 
\begin{align*}
&\calH_{\phi_m}(u(\tau)) \leq
\ee^{-\Lambda}\,\calH_{\phi_m} (u(0)) \quad \text{for all } \tau >0,  
\\
&\big\|\sqrt{ u(\tau,\cdot)} - \sqrt{U(\cdot)} \big\|^2_{\rmL^2(\R)} \leq 
 \ee^{-\Lambda \tau }\,\wh C_m\, \calH_{\phi}(u(0)) \text{ \ for all
 } \tau>0,
\end{align*}
where $\Lambda = \dfrac12\big( 1- m(m{-}1)^2 \Sigma_{U,m} \big)$ and
$\wh C_m:= \sup\bigset{(\sqrt r-1)^2/\phi_m(r)}{ 0<r\neq 1} <\infty$. 
\end{theorem}
\noindent\begin{proof} We proceed exactly as in the previous proof. The choice
  $\phi_m=\varphi_{p_m,q_m}$ yields 
\[
\phi''_m(\rho) \frac{m(\rho^{m-1}{-}1)^2}{4\,\rho^{m-3}} \leq
\frac{m(m{-}1)^2}{2} \, \phi_m(\rho). 
\]
With this and the definition of $\Sigma_{U,m}$, we arrive at 
\[
\frac{\rmd}{\rmd \tau} \calH_{\phi_m}(u(\tau)) \leq  - \frac12 \calH_\phi(u) + \int_\R
\frac{m(m{-}1)^2}{2} \, \phi_m(\rho) 
\Sigma_{U,m} \,U(y) \dd y = - \Lambda \calH_{\phi_m}(u(\tau)). 
\]
This proves the first estimate, and the second follows by the definition of $\wh C_m$.
\end{proof}

Next we show that the second 
condition on $U$ imposed in \eqref{eq:PME.Sigma.Cond} can be controlled by the
estimates obtained in Section \ref{se:ScalProfiles}. With $Q=A'(U)U'=m U^{m-1}
U'$ we have
\[
\Sigma_{U,m}=\sup_{y\in \R}\: (U')^2 U^{m-2} = \sup_{y\in \R}\: \frac{Q^2}{m^2
  U^m}\leq \frac{D^*(U_+{-}U_-)^2/8}{m^2 U_-^m} 
\leq \frac{U_+^{m-1}}{8m U_-^{m} } \,\big(U_+- U_-\big)^2,   
\]
where we used the monotonicity $U_-\leq U(y) \leq U_+$ from Lemma
\ref{le:ScalMonot} and the flux estimates $Q(y)\leq Q(0)$ and $Q(0)^2 \leq
D^*(U_+{-}U_-)^2/8$ from Proposition \ref{pr:Estim.unif.eps} and Corollary
\ref{co:BoundsQ0U0}, respectively. In particular, we see that for the linear
case $m=1$ there is no restriction at all, whereas for $m>1$ there is always a
range of $0<U_-<U_+$ that is allowed including the constant case 
arising from $U_+=U_->0$. \medskip

We now return to the general case of a monotone relation $u\mapsto A(u)$ and
observe that the integral relations \eqref{eq:Rel.U.olUpm} for the similarity
profiles $U$ obtained in Theorem \ref{th:VectValProf} lead to simple relations
for all solutions $u$ of the diffusion equation \eqref{eq:PDE1D}, namely
\begin{align*}
& \frac\rmd{\rmd \tau} \int_\R\!\big( u(\tau,y) - \ol u_\pm(y) \big) \dd y = -
\frac12 \int_\R\!\big( u(\tau,y) - \ol u_\pm(y) \big) \dd y,
\\
& \frac\rmd{\rmd \tau} \int_\R\!y \big( u(\tau,y) - \ol u_\pm(y) \big) \dd y =
 A(U_+) -A(U_-) -  \int_\R\! y \big( u(\tau,y) - \ol u_\pm(y) \big) \dd y. 
\end{align*}
Moreover, the linearization of the diffusion equation \eqref{eq:PDE1D}
around $u=U$ leads to the equation 
\[
v_\tau = \bbL_U\,v:= \big(A'(U)v\big)'' + \frac y2 v', \quad v(\pm\infty)=0.
\] 
It can be easily checked that the linear operator $\bbL_U$ has the eigenvalues
$\lambda_1=-\frac12$ and $\lambda_2=-1$ with the corresponding eigenfunctions
$V_1(y)= U'(y)$ and $V_2(y)=y U'(y)$, namely
\[
\bbL_U U'=-\frac12 U' \quad\text{and}\quad \bbL_U (y U')= -y U'.
\]
Finally, the adjoint operator $\bbL^*_U$ is given via $\bbL^*_U w =
A'(U)w'' - \big(\frac y2 w\big)'$ and has the corresponding eigenfunctions
$W_1(y)=1$ and $W_2(y)=y$. 

Based on this, we conjecture that the global stability obtained in Theorem
\ref{th:PME.Cvg} can be improved via a local
stability analysis. Without loss 
of generality we can assume that the initial condition $u(0,\cdot)$ satisfies 
\[
\int_\R\!\big( u(0,y) - \ol u_\pm(y) \big) \dd y = 0 
\quad\text{and} \quad 
 \int_\R\! y \big( u(0,y) - \ol u_\pm(y) \big) \dd y =  A(U_+) -A(U_-).
\]
This can always be achieved by a suitable translation and scaling $\wt
u(0,y)=u(0,\mu (y{-}y_0))$. Then, for initial conditions $u(0)$ close to $U$ 
one can expect a decay with a decay rate
  close to the third eigenvalue $\lambda_3<\lambda_2=-1$, i.e.\ 
\[
\| u(\tau)-U\|_Y \leq C\ee^{-(1+\delta)\tau} \| u(0){-}U\|_Y \quad \text{for
  all }\tau>0
\]
whenever $ \| u(0){-}U\|_y \leq \delta\ll 1$, where $Y$ is a suitably chosen
Banach space. 

Such results could then be transformed back into the physical variables $t=
\ee^{\tau} -1$ and $x= \ee^{\tau/2} y$ for obtaining algebraic decay, e.g.\ in
$\rmL^1(\R_x)$, see \cite{VanPel77ABSN,Bert82ABSN} for related results in this
direction.

\section{Diffusion systems with reaction constraints}
\label{se:ApplRDS}

We consider one-dimensional RDS for nonnegative concentration vectors
$\bfc(t,x)\in {[0,\infty[}^{i_*}$ with reaction of mass-action law satisfying the
detailed-balance condition 
\begin{equation}
  \label{eq:RDSgen}
  \bfc_t = \big( \bbD(\bfc) \bfc_x \big)_x + \bfR(\bfc)  \quad \text{with }
\bfR(\bfc)= \sum_{r=1}^{r_*}k_r \big(\frac{\bfc^{\bfalpha^r}}{\bfw^{\alpha^r}}
- \frac{\bfc^{\bfbeta^r}}{\bfw^{\beta^r}} \big) \Big( \bfbeta^r-\bfalpha^r\Big)
\in \R^{i_*}, 
\end{equation}
where $\bfw=(w_i)_i\in {]0,\infty[}^{i_*}$ denotes the nonnegative equilibrium
state. Using parabolic coordinates
\[
y=x/\sqrt{t{+}1} \quad \text{and} \quad \tau=\log(t{+}1),
\]
the transformed system takes the form
\begin{equation}
  \label{eq:RDS.scaled}
  \bfc_\tau = \big( \bbD(\bfc) \bfc_y \big)_y+ \frac y2 \bfc_y  +\ee^\tau
  \bfR(\bfc), 
\end{equation}
where now an exponential factor occurs in front of the reaction terms because
they do not scale in the same way as the parabolic terms. We leave the analysis
of the model involving the growing term $\ee^\tau$ to the work
\cite{MieSch22?CSSP} and restrict here to the simpler case with full
invariance. 

A scaling invariant problem is obtained by setting $\ee^\tau$ formally to
$+\infty$, i.e.\ we assume that the reactions $\bfR$ are already in
equilibrium, whereas the spatial diffusion is much slower and still allows for
diffusive fluxes, see the constrained system \eqref{eq:scRDS.Constrain} below.

\subsection{The formally reduced system with reaction constraint}
\label{su:ReducedSyst}

In the parabolic scaling the reactions must be considered very fast. In
a first non-rigorous approximation we can follow the standard argument in
chemical modeling and assume that for all $\tau$ and $y$ the concentration
vector $\bfc(\tau,y)$ is always in equilibrium (i.e.\ $\bfR(\bfc(\tau,y))=0$),
but the equilibrium may still depend on $\tau$ and $y$ and will equilibrate
spatially by diffusion only. We refer to \cite{Both03ILRC, MiPeSt21EDPC,
  PelRen20?FRLG, Step21CGED} for some recent works justifying the limit of
infinitely fast reactions in slow-fast systems.

To describe the set of all equilibria, we assume that the dimension $\gamma_*$
of the stoichiometric subspace 
$\Gamma:=\mafo{span}\bigset{\bfbeta^r-\bfalpha^r}{r=1,...,r_*}$ is less than $i_*$,
which implies that $\bfR(\bfc)=0$ has a nontrivial family of solutions, which
contains the equilibrium $\bfw$. By arguments from standard linear algebra we
can construct  a matrix $\bbQ \in \R^{m_*\ti i_*}$ with $m_*=i_*-\gamma_*$, such that
$\Gamma=\mafo{ker}\,\bbQ$ and $\mafo{range}\,\bbQ^\top = \Gamma^\perp$. Thus, by
construction we have $\bbQ\bfR(\bfc)=0$ for all $\bfc$.

We introduce the so-called slow variables via
\[
\bfu(t,x)=\bbQ \, \bfc(t,x). 
\]
Following \cite{MieSte20CGED, MiPeSt21EDPC}, the detailed-balance condition can
be exploited to characterize the set of all steady
states. For this we introduce
\[
\mfC:={[0,\infty[}^{i_*}, \quad  \mfU:= \bbQ \mfC, \quad 
\calE(\bfc):= \sum_{i=1}^{i_*} w_i\LB(c_i/w_i) \text{ with } \LB(r)=r\log r - r+1.
\]
The following is shown in \cite[Sec.\,3.3]{MiPeSt21EDPC}:
\begin{align*}
   \bigset{\bfc\in \mfC }{ \bfR(\bfc)=0} &=
\mafo{clos}\big(\bigset{\bfc \in {]0,\infty[}^{i_*} }{\exists\, \mu\in
  \R^{m_*}{:}\  \big(\log(c_i/w_i)\big)_i =\bbQ^\top \mu}  \big)  \\
& = 
\Bigset{\bfc\in \mfC}{ \exists\,\bfu\in \mfU{:}\ \bfc \text{
      minimizes }\calE \text{ subject to } \bbQ\bfc=\bfu }.
\end{align*}
Moreover, it is shown in \cite[Prop.\,3.6]{MiPeSt21EDPC} that there is a
continuous map $\Psi:\mfU\to \mfC$ such 
that 
\[
\Psi(\bfu)=\mafo{arg\;\!min}\Bigset{ \calE(\bfc) }{\bbQ \bfc =\bfu} \quad \text{and} \quad
\bfu= \bbQ\Psi(\bfu).  
\] 
The last relation is a direct consequence from the definition of $\Psi$.

Returning to the parabolically scaled RDS in \eqref{eq:RDS.scaled}, the
exponentially growing prefactor $\ee^\tau$ forces the reactions to equilibrate
very fast, see \cite{GalSli22DREE} where the rate $(1{+}t)^{-1/2}=\ee^{-\tau/2}$ is
established.  Thus we may assume that for $\tau\gg 1$ we always have
$\bfc(\tau,y)\approx \Psi(\bfu(\tau,y))$. Moreover we may apply the linear operator
$\bbQ$ to the equation and, using $\bfu=\bbQ\Psi(\bfu)$ we obtain the reduces problem
\begin{equation}
  \label{eq:DiffSys.scaled}
  \bfu_\tau = \Big( \bbQ \:\bbD\big(\Psi(\bfu)\big)\: \Psi(\bfu)_y \Big)_y +
  \frac y2 \bfu_y \quad 
  \text{for }\bfu(\tau,y)\in \mfU\subset \R^{m_*}.
\end{equation}
Note that the reactions have disappeared completely because of
$\bbQ\bfR\equiv 0$, but we also have assumed that they are equilibrated, i.e.\
$\bfc=\Psi(\bfu)$. In summary, we are left with a pure diffusion problem in
parabolic scaling variables. 

This system can equivalently be formulated in terms of the original
concentration vector $\bfc$ as follows:
\begin{equation}
  \label{eq:scRDS.Constrain}
  \bfc_\tau= \big(\bbD(\bfc) \bfc_y \big)_y + \frac y2 \bfc_y + \bflambda,
  \quad \bbQ\bflambda =0, \quad \bfR(\bfc)=0 \qquad \text{for }\tau>0,\ y\in \R. 
\end{equation}
Here $\bflambda \in \Gamma$ arises via the limit $\ee^\tau \bfR(\bfc) \to$
``$\infty \, \bm0$'' and is, thus, a remainder of the much faster reactions. 
Mathematically  $\bflambda \in \Gamma$ can be understood as a Lagrange
multiplier corresponding to the algebraic constraint $\bfR(\bfc)=0$.\medskip

If $\bbD$ is independent of $\bfc$, we can define $\bfA(\bfu)=
\bbQ\bbD\Psi(\bfu)$ and observe that steady states of \eqref{eq:DiffSys.scaled}
have to satisfy our profile equation 
\[
\bfA(\bfU)'' + \frac y2 \bfU'=0 \ \text{ for } y\in \R, \qquad \bfU(y)\to
\bfU_\pm \text{ for } y\to \pm \infty.
\]
If we find such a profile $\bfU:\R\to \mfU \subset \R^{m_*}$, it gives rise to a
stationary profile $\bfC:\R \to \mfC\subset \R^{i_*}$ via
$\bfC(y):=\Psi(\bfU(y))$ which then is a steady state of the constrained diffusion
system \eqref{eq:scRDS.Constrain}.
In \cite{MieSch22?CSSP} cases are discussed in which it is possible to
show that all solutions $\bfc(\tau,\cdot)$ of the full scaled reaction-diffusion
system \eqref{eq:RDS.scaled} satisfying $\bfc(0,y) \to
\Psi(\bfU_\pm)$ for $y \to \pm\infty$ converge to $\bfC$
for $\tau\to \infty$. 

In this work we restrict the discussion to the existence
question for the self-similar profiles $\bfU:\R\to \mfU\subset \R^{m_*}$ and
hence of $\bfC: \R \to \mfC\subset \R^{i_*}$. The profiles $\bfC$ provide exact
self-similar solutions to the unscaled constrained system   
\begin{equation}
  \label{eq:unRDS.Constrain}
  \dot \bfc= \big(\bbD\, \bfc_x \big)_x  + \bflambda,
  \quad \bbQ\bflambda =0, \quad \bfR(\bfc)=0 \qquad \text{for }t>0,\ x\in \R. 
\end{equation}
Because of the nonlinear constraint $\bfR(\bfc)=0$, this is a quasilinear
system. 

In light of the analysis in \cite{GalSli22DREE} and \cite{MieSch22?CSSP}, it is
to be expected that the solutions of the full reaction-diffusion system
\eqref{eq:RDSgen} with the additional boundary conditions $\bfc(t,\pm\infty)=
\Psi(\bfU_\pm)$ behave asymptotically self-similar as well. But such results are
beyond the scope of this work.  

\subsection{Linear reaction-diffusion systems}
\label{su:LinearRDS}

We consider a linear reaction-diffusion system of the form 
\[
\bfc_t = \bbD \bfc_{xx} + \bbB \bfc \quad \text{for }\bfc(t,x) \in \mfC \subset \R^{i_*}.
\]
Here $\bbD=\mafo{diag}(d_i)_{i=1,...,i_*}$ and $\bbB$ is obtained from a
detailed-balance system as in \eqref{eq:RDSgen}, i.e.\ all stoichiometric
vectors $\bfalpha^r$ and $\bfbeta^r$ are unit vectors
$\bigset{\bfe_j}{j=1,...,i_*}$. It is shown in \cite[Sec.\,2]{MieSte20CGED}
that the operator $\bbQ\in \R^{m_*\ti i_*}$ can be constructed such that each
column is a unit vector and that $\Psi(\bfu)=\bbN \bfu$ with $\bbN\in \R^{i_*\ti m_*}$
have nonnegative entries such that each column sum equals 1. In particular, one
has 
\[
\bbQ\bbB=0, \qquad \bbQ \,\bbN = I_{m_*} \in \R^{m_*\ti m_*} \quad \text{and} \quad
\bbN\,\bbQ \in \R^{i_*\ti i_*} \text{ is a projection}.
\]
Thus, for this special case the reduced RDS for $\bfu=\bbQ\bfc$ in scaling
coordinates takes the 
form
\[
 \bfu_\tau = \bbA \bfu_{yy} + \frac y2 \bfu_y \quad \text{with }\bbA = \bbQ
 \bbD\bbN. 
\]
If $\bbD$ is diagonal, then it can be shown that $\bbA$ is also diagonal,
containing the effective diffusion constants for the components $u_m$, cf.\
\cite{Step21CGED}. However, if $\bbD$ is nondiagonal but still positive
semidefinite, then $\bbA$ may non longer be monotone, i.e.\ $\bbA{+}\bbA^\top$
is no longer positive semidefinite. For
instance consider
\[
\bbD=\bma{ccc}d_1&\delta&0\\ \delta&d_2&0\\ 0&0&d_3 \ema ,\ 
\bbQ=\bma{ccc}1&0&0\\0&1&1\ema, \ \bbN=\bma{cc}1&0\\0&\nu\\ 0&1{-}\nu\ema
,\ \bbA= \bma{cc}d_1&\nu\delta\\ \delta &\nu d_2{+}(1{-}\nu)d_3 \ema
\]
with $d_1=d_2=1$, $\delta=1/2$, $\nu= d_3=1/50$. 
Thus, monotonicity is not obtained automatically. 

\subsection{Linear diffusion systems without reactions} 
\label{su:LinSystems}

We consider now linear systems without reactions of the form 
\[
\bfu_t = \bbA \bfu_{xx} \text{ for } (t,x)\in {]0,\infty[}\ti \R, \quad
\bfu(t,\pm\infty) = \bfU_\pm.
\] 
Here $\bbA$ is a monotone matrix, i.e.\ $\bfw\vdot \bbA\bfw\geq \ALPHA |\bfw|^2
\geq 0$. Of course,  the parabolically scaled equation reads 
\begin{equation}
  \label{eq:LinPDE.scal}
\bfu_\tau = \bbA \bfu_{yy} + \frac y2 \bfu_y \text{ for } (t,y)\in
{]0,\infty[}\ti \R, \quad \bfu(t,\pm\infty) = \bfU_\pm.  
\end{equation}
As explained in Example \ref{ex:LinVect}, there is always a unique similarity
profile $\bfU$ for any pair $(\bfU_-,\bfU_+)\in \R^m\ti \R^m$. 

By classical energy estimates using the monotonicity of $\bbA$ one obtains
convergence towards the stationary profile by linearity. If $\bfu$ is a general
solution of \eqref{eq:LinPDE.scal}, then the difference
$\bfw(\tau,y)=\bfu(\tau,y)- \bfU(y) $ is a solution as well. Hence, we obtain
\begin{align*}
\frac12\frac\rmd{\rmd\tau} \|\bfw\|_{\rmL^2}^2 &= \int_\R \bfw\vdot \bfw_\tau =
\int_\R\bfw\vdot\big(\bbA \bfw_{yy}+\frac y2\bfw_y \big) \dd y \\
& = \int_\R\big( {-} \bfw_y \vdot \bbA \bfw_y - \frac14 |\bfw|^2\big) \dd y  \
\leq \ - \frac14 \|\bfw\|_{\rmL^2}^2.
\end{align*}
Thus, Gronwall's estimate yields 
\[
\| \bfu(\tau){-}\bfU\|^2_{\rmL^2} \leq \ee^{-\tau/2} \|
\bfu(0){-}\bfU\|^2_{\rmL^2}.
\]

We emphasize that this estimate is even true in the case of the linear
Schrödinger equation $\ii \psi_t = \psi_{xx}$ which can be realized as a real
system with $\bfu=(\Re\psi,\Im\psi)$ and $\bbA=\binom{0\ -1}{1\ \ 0}$, see
Example \ref{ex:LinVect}(III).

\subsection{One reaction for two species}
\label{su:OneReact}

In \cite{GalSli22DREE, MieSch22?CSSP} the following system of two equations is
studied in detail:
\[
\binom{\dot c_1}{\dot c_2} = \binom{d_1\, \pl_x^2c_1}{d_2 \,\pl_x^2 c_2} 
 - \kappa  \binom{\gamma\,(c_1^\gamma - c_2^\beta)}{\beta\,( c_2^\beta -
   c_1^\gamma)} \qquad \text{for } t>0 \text{ and } x \in \R. 
\]
The two concentrations $c_1, c_2\geq 0$ for the species $X_1,X_2$ 
diffusive with diffusion constants $d_j$ and
undergo the reversible mass-action reaction $\gamma X_1 \rightleftharpoons
\beta X_2$.   

The scaled and constrained system \eqref{eq:scRDS.Constrain} 
takes the form
\[
\pl_\tau \binom{c_1}{c_2} = \binom{d_1\, \pl_y^2c_1}{d_2\, \pl_y^2 c_2} + \frac y2
\pl_y\binom{c_1}{c_2} +  \lambda \binom{\gamma}{-\beta}, \quad \lambda \in \R,
\quad c_1^\gamma = c_2^\beta .
\]
The set of equilibria for $\bfR$ is a one-parameter family given by 
\[
\bigset{\bfc\in \mfC}{ \bfR(\bfc)=0} = \bigset{(A^\beta,A^\gamma)}{ A\geq 0}.
\]

We consider the stoichiometric mapping $\bbQ=\big( \beta\ \ \gamma\big) \in \R^{1\ti 2}$
defining $u=\beta c_1+ \gamma c_2\geq 0$. The function
$\Psi:{[0,\infty[}\to{[0,\infty[}^2$ 
is defined via 
\[
\bfc=\Psi(u)=\binom{\psi_1(u)}{\psi_2(u)}    \quad
\Longleftrightarrow \quad \big(\  u=\bbQ\bfc=\beta c_1{+}\gamma c_2 \text{ and } 
c_1^\gamma  = c_2^\beta \ \big).
\]
The case $ \gamma = \beta $ leads to the simple relation $\Psi(u) = 
\frac1{\beta{+}\gamma} \binom uu$.  
If $ \beta \neq \gamma $, we may assume $\beta < \gamma$ without loss of
generality. Then, 
\begin{align*}
  & \psi_1 \text{ is concave}, && \psi_1 (u)= u/\beta+
  \mafo{h.o.t.}_{u\to 0^+},&& \psi_1 (u) = \big(u/\gamma\big)^{\beta/\gamma} 
   +  \mafo{l.o.t.}_{u\to \infty},
  \\
  & \psi_2 \text{ is convex}, 
  && \psi_2 (u)=\big(u/\beta\big)^{\gamma/\beta} + \mafo{h.o.t.}_{u\to 0^+}, 
 && \psi_2 (u) = u/\gamma +  \mafo{l.o.t.}_{u\to \infty} .
\end{align*} 
For $A_\Psi(u):= \bbQ\bfD \Psi(u)=\binom{\beta d_1}{\gamma d_2} \vdot \Psi(u)$ we can
use $0<\psi'_1 (u) \leq \psi'_1(0)=1/\beta$ and
$0<\psi'_2(u) \leq \psi'_2(\infty) = 1/\gamma$. This yields
\[
\bbD(u)=A'_\Psi(u)\in [D_*,D^*] ,\quad \bbD(u)\to d_1 \text{ for }u\to 0^+, \quad 
\bbD(u)\to d_2 \text{ for } u\to \infty,
\]       
where $D_*=\min\{d_1,d_2\}$ and $D^*=\max\{d_1,d_2\}$.

Thus, the theory of Section \ref{se:ScalProfiles} applies (simply extend $\bbD$
by $\bbD(u)=d_1$ for $u\leq 0$). We are in the
nondegenerate case, where the resulting profiles $U$ solving 
\[
\big(A_\Psi(U)\big)'' + \frac y2 \, U'=0 \ \text{ on } \R, \quad U(\pm\infty)=U_\pm,
\]
are smooth, strictly monotone and converge to its two limits like the error
function. In addition to $U_- \leq U(y) \leq U_+$ the estimate
\begin{equation}
  \label{eq:U'Estim}
  0 \leq U'(y) \leq \ee^{-y^2/(4D^*)} \sqrt{\tfrac{\ds D^*}{\ds 8D_*^2}}\, \big( U_+ -
U_-\big) \quad \text{for all }y \in \R
\end{equation}
holds, even in the
case $U_-=0$, where asymptotically the concentrations vanish, viz.\ 
$\bfC_-=\Psi(U_-)=\binom00$, because the
effective diffusion is still bounded from below by $D_*>0$.    

Of course, a profile $U:\R\to [U_-,U_+]$ for the reduced equation leads to a smooth
concentration profile $\bfC:\R\to \mfC\subset \R^2$ given by
$\bfC(y)=\Psi(U(y))$
and satisfying the profile equation  
\begin{align*}
&0= \bma{cc}d_1&0\\ 0&d_2\ema \bfC''(y)+ \frac y2 \,\bfC'(y) + \Lambda_U(y)
\binom\gamma{-\beta}, \quad C_1(y)^\gamma=C_2(y)^\beta, \\
&
 \bfC(y)\to \Psi(U_\pm) \text{ for }y\to \pm\infty.
\end{align*}

\begin{remark}\label{re:U''+y2U'}
In \cite{MieSch22?CSSP} the convergence to the asymptotic steady state $y
\mapsto \bfC(y)= \Psi(U(y))$ for the scaled reaction-diffusion system
\eqref{eq:RDS.scaled} is investigated. For this, it is necessary to bound the
Lagrange multiplier
\[
\Lambda_{U}(y):=- \frac1{\gamma}\Big(d_1 C''_1(y) + \frac y2 C'_1(y)\Big) =
\frac1{\beta }\Big(d_2 C''_2(y) + \frac y2 C'_2(y)\Big) \ \text{ in } \rmL^\infty(\R),
\]
where the second identity holds by construction
from $A_\Psi(u)=\beta d_1 \psi_1(u)+\gamma d_2 \psi_2(u)$. Using the relation 
$C_1(y)= \psi_1(U(y))$, where $\psi_1:{]0,\infty[} \to {]0,\infty[}$ is
$\rmC^\infty$, and exploiting the bounds $0<U_- \leq U(y) \leq U_+$, the identity
$U''=-\big( A''(U) (U')^2 + \frac y2 U'\big)/A'(U)$, and estimate
\eqref{eq:U'Estim}, we obtain the following result.  Fixing $d_1,d_2>0$ and
$\gamma > \beta>0$, for every $M>0$ there exists a constant $C_M>0$ such that
\[
U_+,U_-\in [1/M,M] \text{ \ implies} \quad 
\big| \Lambda_{U}(y) \big| \leq C_M \big|U_+-U_-\big| \text{ \ for all }
y\in \R. 
\] 
\end{remark}
\subsection{One reaction for three species}
\label{su:OneReact3Spec}

We consider the classical binary reaction $X_3 \rightleftharpoons  X_1 {+} X_2$ leading
to the scaled {\setlength{\arraycolsep}{0.12em} RDS 
\[
\pl_\tau \bfc 
 = \bfD \pl_{y}^2 \bfc   
   + \frac y2  \pl_y \bfc 
  - \ee^\tau\,\big(c_1c_2 {-} c_3\big) \ThreeVect11{-1}\quad \text{with } \bfD
  = \bma{ccc} d_1&0&0\\[-0.2em] 0&d_2&0\\[-0.2em] 0&0&d_3\ema.
\]
The} associated profile equation reads
\begin{equation}
  \label{eq:ProfEqn3Spec}
  \bfD \bfC''(y) + \frac y2  \bfC'(y) + \lambda(y) \ThreeVect11{-1}  =0, \ 
    C_1(y)C_2(y)=C_3(y) \ \text{ and } \
    \bfC(\pm\infty) = \Psi(\bfU_\pm). 
\end{equation}

The set of equilibria for $\bfR$ is a two-parameter family, namely
\[
\bigset{\bfc\in \mfC}{ \bfR(\bfc)=0} = \bigset{(A,B,AB)}{ A, B\geq 0}.
\]
We can choose the stoichiometric matrix 
\[
\bbQ = \bma{ccc} 1&0&1\\ 0&1&1\ema \in \R^{2\ti 3} 
\]
and obtain $\bfu=\binom{u_1}{u_2} = \bbQ\bfc \in \mfU:={[0,\infty[}^2$. The reduction
function $\Psi:\mfU\to \mfC$ can be  calculated explicitly in the form
\[
\Psi(u_1,u_2)=\frac12
  \ThreeVect{ u_1{-}u_2{-}1 + s(\bfu) }
            { u_2{-}u_1{-}1 + s(\bfu) }
            { u_1{+}u_2{+}1 - s(\bfu) } \quad
   \text{with  } s(\bfu):=\sqrt{(1{+}u_1{+}u_2)^2-4u_1u_2}.
\]
To extend $s$ to a function $s:\R^2\to \R$ we simply set $s(u_1,u_2)=1+u_1+u_2$
whenever $u_1\leq 0$ or $u_2\leq 0$ and observe that $s$ is globally Lipschitz
continuous. Moreover, $s_j(\bfu)=\pl_{u_j} s(\bfu)$ satisfies $s_1(\bfu)\leq
1$, $s_2(\bfu)\leq 1$ and $s_1(\bfu) +s_2(\bfu)\geq 0$ for all $\bfu\in \R^2$. 

From this we can calculate the function $\bfA(\bfu)=\bbQ\bfD\Psi(\bfu)$ with
$\bfD=\mafo{diag}(d_j)$:
\[
\bfA(\bfu)= \frac12 \binom{(d_1{+}d_3) u_1 + (d_3{-}d_1)(1{+}u_2{-}s(\bfu)) }
  { (d_2{+}d_3) u_2 + (d_3{-}d_2)(1{+}u_1{-}s(\bfu)) } .
\]
For general $\rmC^1$ functions $\bfA$ we have the equivalence
\[
\forall\, \bfu,\wt\bfu:\ \langle\bfA(\bfu){-}\bfA(\wt\bfu), \bfu{-}\wt\bfu \rangle \geq
\ALPHA |\bfu{-}\wt\bfu |^2 \quad \Longleftrightarrow \quad 
\forall\, \bfu:\ \frac12\big(\rmD\bfA(\bfu){+}\rmD\bfA(\bfu)^\top\big)\geq \ALPHA
I_{m\ti m}.
\]
Abbreviating $s_j=\pl_{u_j}s(\bfu)$ and $\delta_j=1{-}d_j/d_3$ for $j=1,2$ we find 
\[
\frac12\big(\rmD\bfA(\bfu){+}\rmD\bfA(\bfu)^\top\big)= \frac{d_3}2
\bma{cc} 2- \delta_1 -\delta_1 s_1   & 
\frac12\big(\delta_1{+}\delta_2{-} \delta_1s_2 {-} \delta_2s_1  \big)\\
\frac12\big(\delta_1{+}\delta_2{-} \delta_1s_2 {-} \delta_2s_1 \big)&
2-\delta_2 - \delta_2s_2 \ema 
=:\bfG .
\]

For $d_1=d_2=d_3$ we have $\delta_1=\delta_2=0$ and obtain $\bfG=d_3 I_{2\ti
  2}$ giving monotonicity with $\ALPHA= d_3>0$. Using $s_j\in
[-1,1]$ it is also easy to show that $|\delta_1|,|\delta_2|<1/2$ is sufficient
for showing that $\bfG$ is positive definite. More precisely, we have the
following result.

\begin{lemma}[Monotonicity]
\label{le:Monot.bfA}
The function $\bfA: \R^2\to \R^2$ is strictly monotone ($\exists\,
\ALPHA>0$ $\forall\, \bfu,\bfw\in \R^2{:}$ $\langle \bfA(\bfw){-}\bfA(\bfu),
\bfw{-}\bfu\rangle \geq \ALPHA |\bfw{-}\bfu|^2$) if and only if 
\[
(3{-}\sqrt8\,)d_3 < d_j < (3{+}\sqrt8\,) d_3 \quad \text{ for }j=1,\ 2 .
\]  
\end{lemma} 
\begin{proof} We keep $\delta_1$ and $\delta_2$ fixed and 
observe that $\mu(s_1,s_2):=\det \bfG$ is a quadratic polynomial in
$(s_1,s_2)$ which is concave, as the quadratic terms can be combined to
$-(\delta_2 s_1{-}\delta_1 s_2)^2$. As all $(s_1,s_2)$ lie in the triangle
$S:=\bigset{(s_1,s_2)}{ s_1\leq 1, \ s_2\leq 1, s_1{+}s_2\geq 0}$ the assertion
follows if we can show that $\min\bigset{\mu(s_1,s_2) }{ (s_1,s_2)\in S} $ is
positive. By concavity the minimum is attained in one of the three corners
because they are the extremal points.

We have $\mu(1,1)>0$ whenever $\delta_1<1$ and $\delta_2<1$. Moreover,
$\mu(1,-1)>0$ holds for $|\delta_1{+}2|< \sqrt8$ and $\mu(-1,1)>0$ holds for
$|\delta_2{+}2|< \sqrt8$.  Inserting $\delta_j=1-d_j/d_3$, the desired result
follows.
\end{proof}

Under the assumptions of the above monotonicity result, our existence theory in
Theorem \ref{th:VectValProf} provides unique similarity profiles
$\bfU:\R \to \R^2$ connecting $\bfU_-$ and $\bfU_+$. These solutions give rise
to similarity profiles $\bfC = \Psi{\circ} \bfU$ connecting $\Psi(\bfU_+)$ and
$\Psi(\bfU_+)$ if and only if $\bfU(y) \in \mfU={[0,\infty[}^2$ for all $y \in \R$, thus
providing $\bfC(y)=\Psi(\bfU(y)) \in \mfC={[0,\infty[}^3$. In general
we cannot guarantee this condition, but Corollary \ref{co:UniformBdd} provides an
estimate of the form
\[
\big| \bfU(y) - \wt\bfu_\pm(y) \big| \leq C_* |\bfU_+{-} \bfU_-| = C_*  \Delta_\pm,
\]
where $C_*$ only depends on $d_1$, $d_2$, and $d_3$, but not on $\bfU_\pm$.  As
$\wt\bfu_\pm(y) $ takes values on the straight line connecting $\bfU_-$ and
$\bfU_+$, we conclude that our abstract theory is applicable if
$ (3{-}\sqrt8)d_3< d_1,d_2< (3{+}\sqrt{8})d_3 $ and $|\bfU_+{-} \bfU_-|$ is
sufficiently small compared to the distance of $\bfU_+$ and $\bfU_-$ from the
boundary of $\mfU$. Then similarity profiles $\bfC:\R \to \R^3$ solving
\eqref{eq:ProfEqn3Spec} exist and are unique.

In the present example we obtain nonmonotone profiles $\bfC:\R\to
\mfC\subset\R^3$. For this, consider the case $d_1=d_2$ and the limits 
\[
\bfC_-=(A,B,AB)^\top \quad \text{and} \quad
\bfC_+= (B,A,AB)^\top\quad \text{ with } A\neq B. 
\]
Our uniqueness result and the reflection symmetries $x \to -x$ and
$(c_1,c_2) \to (c_2,c_1)$ imply that the stationary profile $\bfC$ satisfies
$C_1(y)=C_2(-y)$ and $C_3(y)=C_3(-y)$. Using $C_1(y)C_2(y)=C_3(y)$ for all
$y\in \R$ we see that $C_3$ cannot be constant, hence it must be
nonmonotone. In Figure \ref{fig:Nonmonotone} we show an example.
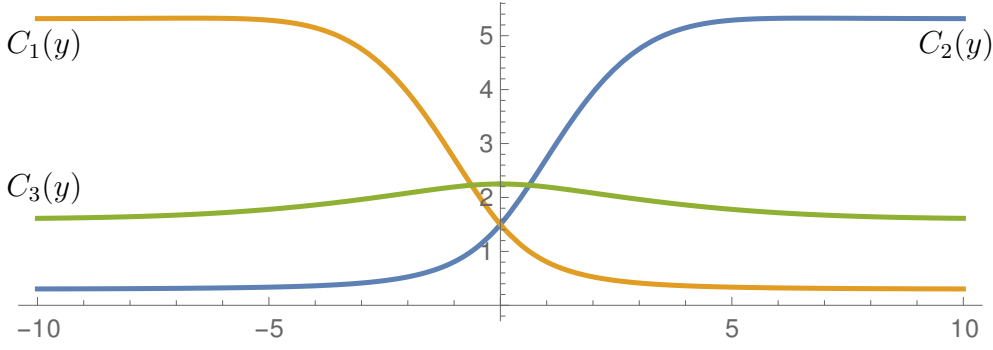
\begin{figure}
\begin{tikzpicture}
\node (PP) at (0,0){\includegraphics[width=0.8\textwidth]{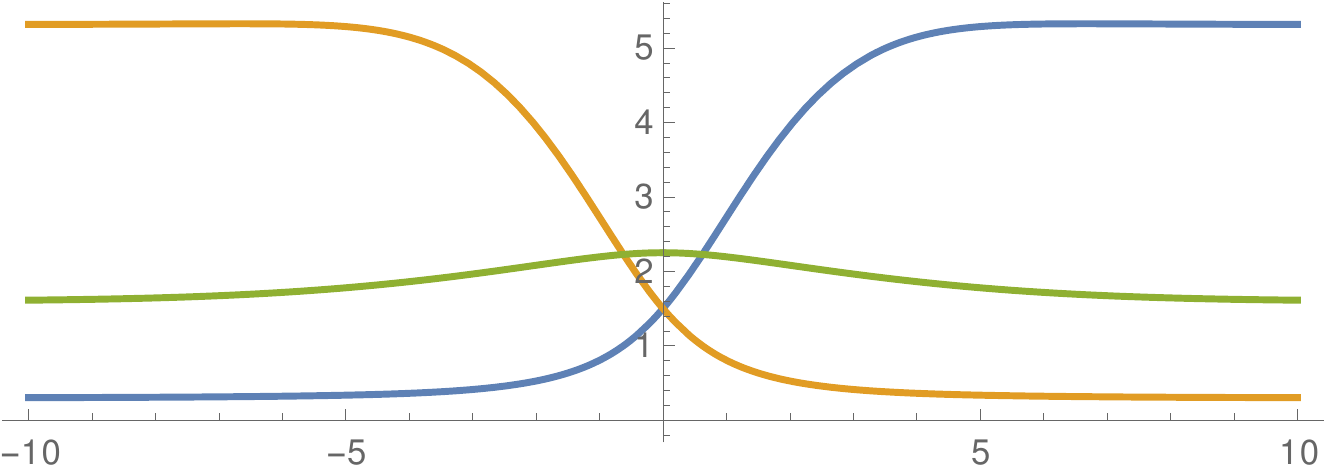}};
\node at (-6,1.7) {$C_1(y)$};
\node at (+6,1.7) {$C_2(y)$};
\node at (-6,-0.2) {$C_3(y)$};
\end{tikzpicture}
\caption{Solution $\bfC=(C_1(y),C_2(y),C_3(y))$ of \eqref{eq:ProfEqn3Spec} for
  $d_1=d_2=2$ and $d_3=10$ with limiting values $\bfC_-\approx(5.3,0.3,1.6)$ and 
$\bfC_+\approx(0.3,5.3,1.6) $. This symmetric solution was obtained by starting with
$\bfC(0)=(1.5,1.5,2.25)$ and $\bfC'(y)=(-1,1,0)$.}
\label{fig:Nonmonotone}
\end{figure} 

An interesting open question is whether there is a stationary profile $\bfC$
connecting the limiting cases 
\[
\bfC_-=\Psi(1,0)=(1,0,0)^\top  \quad \text{and} \quad
\bfC_-=\Psi(0,1)=(0,1,0)^\top .
\]
The profile would see only one of the species $X_1$ or $X_2$ in the limits to
$\pm\infty$, however in the middle region all three species must be present to allow
the generation of the other species.

\subsection{Two reactions for three species}
\label{su:TwoReact3Spec} 

Consider the two reactions $2X_1\rightleftharpoons X_2$ and  $  X_2
\rightleftharpoons X_3$ giving 
\begin{equation}
  \label{eq:2React3Spec}
  \pl_\tau \bfc = \bfD \pl_{y}^2 \bfc
  -k_1\big(c_1^2-c_2\big)\ThreeVect2{-1}0 
  - k_2  \big(c_2-c_3\big) \ThreeVect01{-1}. 
\end{equation}
The set of equilibria is the one-parameter family given by 
\[
\bigset{\bfc\in \mfC}{ \bfR(\bfc)=0} = \bigset{(A,A^2,A^2)}{ A\geq 0}.
\]
Note that the RDS system has invariant regions of the form $\Sigma:= [b,B]\ti
[b^2,B^2] \ti [b^2,B^2]$ for arbitrary $0\leq b < B <\infty$, see
\cite[Chap.\,14\,§B]{Smol94SWRD}. This means that any solution satisfying $\bfc(0,x)\in
\Sigma$ for all $x\in \R$  also satisfies $\bfc(t,x)\in \Sigma$ for all $t>0$
and $x \in \R$. Thus, a similarity profile connecting $\bfC_-=(b,b^2,b^2)$ and
$\bfC_+=(B,B^2,B^2)$ is expected to lie in the invariant region $\Sigma$. 

The stoichiometric matrix is $\bbQ = (1\ \ 2 \ \ 2) \in \R^{1\ti 3}$ and 
\[
u= \bbQ \bfc = c_1 + 2 c_2+ 2 c_3 \quad \text{yields }
\Psi(u)=\ThreeVect{\sigma(u)}{(u{-}\sigma(u))/4}{(u{-}\sigma(u))/4} \text{ with
  } \sigma(u)=(\sqrt{1{+}16u}-1)/8.
\]
With $\sigma'(u)=1/\sqrt{1{+}16u}\in [0,1]$ we easily see that all mappings $u
\mapsto \Psi_j(u)$ are monotonously increasing. Moreover, the function $A(u)=
\bbQ\bfD \Psi(u)$ satisfies
\[
A(u)=  \tfrac{d_2{+}d_3}2 \,u + \big( d_1{-}
\tfrac{d_2{+}d_3}2\big) \,\sigma(u) \quad \text{and } 
\min\big\{ d_1, \tfrac{d_2{+}d_3}2\big\} \leq  A'(u) \leq 
\max\big\{ d_1, \tfrac{d_2{+}d_3}2\big\}.
\]
Thus, the scalar theory of Section \ref{se:ScalProfiles} is applicable and for
$0\leq U_- \leq U_+ <\infty$ there exists a unique similarity profile $U\in
\rmC^\infty(\R; [U_-,U_+])$ that is monotonously increasing.  

As a consequence, the profile equation  
\begin{equation}
  \label{eq:ProfEqn3S2R}
\begin{aligned}  
 \bfD \bfC''(y) + \frac y2  \bfC'(y) + &\lambda_1(y) \ThreeVect2{-1}0
  +\lambda_2(y) \ThreeVect01{-1} =0, \\ 
&    C_1(y)^2=C_2(y)=C_3(y) \ \text{ and } \
    \bfC(\pm\infty) = \ThreeVect{B_\pm}{B_\pm^2}{B_\pm^2}
\end{aligned}
\end{equation}
has for all $B_-\leq B_+$ a unique solution $\bfC$ and each component $C_j$ is
monotonously increasing, and hence lying in the invariant region
$\Sigma=[B_-,B_+]\ti [B_-^2,B_+^2] \ti [B_-^2,B_+^2]$.

\section{Various systems with similarity profiles}
\label{se:VariousSyst}

In this section we mention the connection of our theory to two more systems in
which similarity profiles play a nontrivial role.  The first example concerns the
diffusive mixing between role patterns in the real Ginzburg-Landau equation as
studied in \cite{BriKup92RGGL,GalMie98DMSS}. The second example is a system of
two degenerate parabolic equations that are coupled to satisfy a
thermodynamical conservation law.

\subsection{Profiles connecting roles in the Ginzburg-Landau equation}
\label{su:GinzburgLandau}

For a complex-valued amplitude $Z(t,x) \in \C$ the real 
Ginzburg-Landau equation (i.e. the coefficients are real)
\begin{equation}
  \label{eq:RGLeqn}
  \dot Z = Z_{xx} + Z - |Z|^2 Z
\end{equation}
is an important model in bifurcation theory and pattern formation. It has an
explicit two-parameter family of steady state pattern  in form of the  
role solutions $Z(x)=U_{\eta,\varphi}(x):=\sqrt{1{-}\eta^2} \:\ee^{\ii (\eta
  x+\varphi)}$ with wave 
number $\eta \in [-1,1]$ and phase $\varphi\in [0,2\pi]$.

Starting from \cite{BriKup92RGGL,ColEck92SPSG}, it was shown in
\cite{GalMie98DMSS} that asymptotically self-similar profiles exist that
connect two different role solutions $U_{\eta_-,\varphi_-}$ at $x\to -\infty$
and $U_{\eta_+,\varphi_+} $ at $x\to \infty$. Indeed, the monotone operator
approach used in Theorem \ref{th:VectValProf} for solving the profile equation
was initiated in \cite[Thm.\,3.1]{GalMie98DMSS}.

Writing $Z=r\ee^{\ii u}$ and assuming $r(t,x)>0$, the real Ginzburg-Landau
equation can be rewritten as the coupled system $\dot r = r_{xx}  + r\big(
1{-}r^2 {-} u_x^2\big)$ , $\dot u = u_{xx} + 2\frac{r_x}r u_x$.
Assuming $r^2 + u_x^2 \approx 1$ for $t\gg 1$, one is lead to the
so-called \emph{phase diffusion equation} 
\[
\dot u =\big( A(u_x)\big)_x =A'(u_x)u_{xx} , \quad \text{where }A'(\eta)
 = \frac{ 1 {-}3\eta^2 }{ 1{-}\eta^2}  .
\] 
Introducing the local wave number $\eta(t,x)=u_x(t,x)$ one finds the
quasilinear equation 
\[
\dot \eta = \big(A(\eta)\big)_{xx} \quad \text{with } 
 A'(\eta)  >0 \text{ for } \eta\in
{\Big]\frac{-1}{\sqrt3}, \frac1{\sqrt3} \Big[}. 
\]
The existence of self-similar profiles $\ol \eta:\R\to [\eta_-,\eta_+]$ 
connecting $\eta_-$ and $\eta_+$ (where $-1/\sqrt3 < \eta_- \leq \eta_+ <
1/\sqrt3$) as well as the local convergence of the full solutions $Z(t,x)$
of  \eqref{eq:RGLeqn} to the corresponding asymptotic  
profile $\sqrt{1-\ol \eta(x/t^{1/2})}\: \ee^{\ii\, t^{1/2}\, H(x/t^{1/2})} $ with
$H'(y)= \ol \eta(y)$ is established in \cite{GalMie98DMSS} using suitable
weighted Sobolev norms.

\subsection{A coupled system motivated by thermodynamics}
\label{su:CoupledSyst}

The following degenerate parabolic system couples a velocity-like variable $v$
to an energy-like variable $w$ such that the total momentum
$\calV(v)=\int_{\R^d} v(x) \dd x$ and the total energy $\calE(v,w)=\int_{\R^d}
\big( \frac12v(x)^2 {+}w(x)\big) \dd x$ are conserved along solutions of 
\begin{subequations}
  \label{eq:SM01}
\begin{align}
  \label{eq:SM01.a}
& \dot v = \DIV\big( \eta(w) \nabla v\big), &\text{for }&(t,x)\in {]0,\infty[} \ti \R^d,
\\
  \label{eq:SM01.b}
&\dot w = \DIV \big( \kappa(w) \nabla
  w\big) + \eta(w) |\nabla v|^2 &\text{for }&(t,x)\in {]0,\infty[} \ti \R^d,
\end{align}
\end{subequations}
see \cite{Miel23TCDP} for more motivation. 
Because of the full invariance under the parabolic scaling, 
the parabolically scaled equation is independent of $\tau$:
\begin{equation}
  \label{eq:TraS1ParaSyst}
 \begin{aligned}
& \pl_\tau \wt v - \frac12\, y{\cdot} \nabla \wt v= \DIV\!\big( \eta(\wt w) \nabla
\wt v\big), \quad 
\pl_\tau \wt w- \frac12\,y{\cdot} \nabla \wt w= \DIV\!\big( \kappa(\wt w)\nabla
\wt w\big) + \eta(\wt w) \big|\nabla \wt v\big|^2.
\end{aligned} 
\end{equation}

As the system contains the porous medium equation \eqref{eq:PDE1D} with
$A(w)=\frac1{\beta+1} w^{\beta+1}$ (by simply
setting $v\equiv 0$) there are the classical Barenblatt solutions
as a steady state $(V,W)=(0,B_M)$ where $M\geq 0$ is the mass $M=\int_{\R^d}
B_M(y) \dd y$. As studied in Section \ref{se:ScalProfiles}, there are also
similarity profiles of the form $(v,w)=(0,W)$, however, we expect that it is
also possible to show that for each pair $(V_\pm,W_\pm)$ with $V_-,V_+\in \R$
and $W_-,W_+\geq 0$ there is a unique similarity profile. However,  it seems that our
monotonicity approach developed in Section \ref{se:VectorValued} cannot be used here. 
  
Nevertheless, a nontrivial explicit self-similar solution can be given in the
case $\eta(w)=\kappa(w)=w$
with the limits $(V_\pm,W_\pm)=(\pm\sqrt 2\,B,0)$ (cf.\
\cite[Ex.\,2.2]{Miel23TCDP}), namely 
\begin{equation}
  \label{eq:ExplicSol01}
  \big(V(y),W(y) \big)= 
\left\{ \ba{cl} \big( y/\sqrt{2}\;,\; B^2 {-}y^2/4\big ) &\text{for }|y|\leq 2B, \\
\big(\pm \sqrt2\,B\;,\; 0\big)& \text{for } \pm y\geq 2B, \ea\right.
\end{equation}

\paragraph*{Acknowledgments.} The research was partially supported by Deutsche
Forschungsgemeinschaft (DFG) via the Collaborative Research Center SFB\,910
``Control of self-organizing nonlinear systems'' (project number 163436311),
subproject A5 ``Pattern formation in coupled parabolic systems''.

\small

\addcontentsline{toc}{section}{References}


\def\cprime{$'$}
\providecommand{\bysame}{\leavevmode\hbox to3em{\hrulefill}\thinspace}
\providecommand{\MR}{}

\end{document}